\documentclass[a4paper,12pt]{amsart}

\usepackage[includeheadfoot,twoside=True]{geometry}
\usepackage{comment}

\usepackage[active]{srcltx}\usepackage{amsmath,amsthm,amsxtra}
\usepackage{amssymb}
\usepackage{mathabx}

\usepackage[mathscr]{eucal}

\usepackage{bm}
\usepackage[all]{xy}

\usepackage{aliascnt}

\theoremstyle{plain}
\newtheorem{thm}{Theorem}[section]
\newtheorem*{thm*}{Theorem}

\newaliascnt{prop}{thm}
\newaliascnt{cor}{thm}
\newaliascnt{lem}{thm}
\newaliascnt{claim}{thm}
\newaliascnt{defn}{thm}
\newaliascnt{ques}{thm}
\newaliascnt{conj}{thm}
\newaliascnt{fact}{thm}
\newaliascnt{rem}{thm}
\newaliascnt{ex}{thm}
\newtheorem{prop}[prop]{Proposition}
\newtheorem{cor}[cor]{Corollary}
\newtheorem{lem}[lem]{Lemma}
\newtheorem{claim}[claim]{Claim}
\newtheorem*{prop*}{Proposition}
\newtheorem*{cor*}{Corollary}
\newtheorem*{lem*}{Lemma}
\newtheorem*{claim*}{Claim}
\theoremstyle{definition}
\newtheorem{defn}[defn]{Definition}

\newtheorem*{defn*}{Definition}
\newtheorem*{ques*}{Question}
\newtheorem*{conj*}{Conjecture}

\newtheorem*{prob*}{Problem}

\newtheorem{rem}[rem]{Remark}
\newtheorem{ex}[ex]{Example}
\newtheorem*{fact*}{Fact}
\newtheorem*{rem*}{Remark}
\newtheorem*{ex*}{Example}
\aliascntresetthe{prop}
\aliascntresetthe{cor}
\aliascntresetthe{lem}
\aliascntresetthe{claim}
\aliascntresetthe{defn}
\aliascntresetthe{ques}
\aliascntresetthe{conj}
\aliascntresetthe{fact}
\aliascntresetthe{rem}
\aliascntresetthe{ex}

\usepackage[pointedenum]{paralist}
\usepackage{varioref}
\labelformat{equation}{\textnormal{(#1)}}
\labelformat{enumi}{\textnormal{(#1)}}

\usepackage[a4paper]{hyperref}

\def\textsectionN~{\textsection{}}

\renewcommand\phi{\varphi}
\renewcommand\epsilon{\varepsilon}
\renewcommand\leq{\leqslant}
\renewcommand\geq{\geqslant}

\newcommand{\Set}[1]{\left\{ #1 \right\}}
\makeatletter
\newcommand{\set}{  \@ifstar{\@setstar}{\@set}}\newcommand{\@setstar}[2]{\{\, #1 \mid #2 \,\}}
\newcommand{\@set}[1]{\{ #1 \}}
\makeatother
\newcommand{\lin}[1]{\langle #1 \rangle}
\newcommand{\trans}[1][1]{\raisebox{#1ex}{\scriptsize\kern0.1em$t$\kern-0.1em}}

\newcommand{\PP}{\mathbb{P}}
\newcommand{\cP}{{\PP_{\!\! *}}}
\newcommand{\TT}{\mathbb{T}}
\newcommand{\PN}{\PP^N}

\newcommand{\A}{\mathbb{A}}
\newcommand{\sO}{\mathscr{O}}
\newcommand{\ZZ}{\mathbb{Z}}

\newcommand{\RR}{\mathbb{R}}

\newcommand{\textgene}[1]{\ \ \text{#1}\ \,}

\newcommand{\textand}{\textgene{and}}

\DeclareMathOperator{\rk}{rk}

\DeclareMathOperator{\im}{im}

\DeclareMathOperator{\Spec}{Spec}

\newcommand{\Gr}{\mathbb{G}}

\newcommand{\tdiff}[2]{{\partial #1}/{\partial #2}}
\newcommand{\diff}{\tdiff}

\newcommand{\KK}{\Bbbk}
\newcommand{\kk}{\Bbbk}

\newcommand{\Kx}{\KK^{\times}}

\newcommand{\tNd}{\tilde{N}'}
\newcommand{\RNj}[1][\tNd]{0 \leq j \leq #1}

\DeclareMathOperator{\rank}{rk}

\DeclareMathOperator{\chara}{char}
\DeclareMathOperator{\Aff}{Aff}

\def\N{\mathbb{N}}
\def\Z{\mathbb{Z}}
\def\Q{\mathbb{Q}}
\def\R{\mathbb{R}}
\def\C{\mathbb{C}}
\def\A{\mathbb{A}}

\def\r+{\mathbb{R}_{\geq 0}}

\def\r+{{\R}_{\geq 0}}
\def\q+{{\Q}_{\geq 0}}

\def\arw{\rightarrow}
\def\*c{\C^{\times}}

\def\A{\mathbb {A}}

\def\C{\mathbb {C}}

\def\G{\mathbb {G}}

\def\N{\mathbb {N}}

\def\Q{\mathbb {Q}}
\def\R{\mathbb {R}}

\def\Z{\mathbb {Z}}

\makeatletter
  
  \@addtoreset{equation}{section}
\makeatother

\title[Gauss maps of toric varieties]{Gauss maps of toric varieties}

\author[K.~Furukawa]{Katsuhisa~FURUKAWA}
\address{
  Department of Mathematics,
  School of Fundamental Science and Engineering,
  Waseda~University,
  3-4-1 Ohkubo, Shinjuku, Tokyo 169-8555, Japan
}
\email{katu@toki.waseda.jp}

\author[A.~Ito]{Atsushi~Ito}
\address{Graduate School of Mathematical Sciences, 
The University of Tokyo, 3-8-1 Komaba, 
Meguro, Tokyo 153-8914, Japan}
\email{itoatsu@ms.u-tokyo.ac.jp}


\begin{document}

\maketitle

\begin{abstract}
  We investigate Gauss maps of (not necessarily normal) projective toric varieties over an algebraically closed field of arbitrary characteristic.
  The main results are as follows:
  (1) The structure of the Gauss map of a toric variety is described in terms of combinatorics in any characteristic.
  (2) We give a developability criterion in the toric case. 
  In particular, 
  we show that
  any toric variety whose Gauss map is degenerate must be the join of some toric varieties in characteristic zero.
  (3) As applications, we provide two constructions of toric varieties whose Gauss maps have some given data (e.g., fibers, images) in positive characteristic.
\end{abstract}

\section{Introduction}
\label{sec:introduction}

Let $X \subset \PN$ be an $n$-dimensional projective variety over an algebraically closed field $\Bbbk$ of arbitrary characteristic.
The \emph{Gauss map} $\gamma$ of $X$
is defined as a rational map
\[
\gamma: X \dashrightarrow \Gr(n, \PN),
\]
which sends each smooth point $x \in X$ to the embedded tangent space $\TT_xX$ of $X$ at $x$ in $\PN$.
The Gauss map is a classical subject and has been studied by many authors.
For example, it is well known
that a general fiber of the Gauss map $\gamma$
is (an open subset of) a linear subvariety of $\PN$ in characteristic zero
(P.~Griffiths and J.~Harris \cite[(2.10)]{GH},
F.~L.~Zak \cite[I, 2.3.\ Theorem (c)]{Zak}).
The linearity of general fibers of $\gamma$ also holds
in arbitrary characteristic
if $\gamma$ is {separable} \cite[Theorem 1.1]{expshr}.
We denote by $\delta_{\gamma}(X)$ the dimension of a general fiber of $\gamma$,
and call it the \emph{Gauss defect} of $X$ (see \cite[2.3.4]{FP}).
The Gauss map $\gamma$ is said to be \emph{degenerate} if $\delta_{\gamma}(X) > 0$.

\vspace{2mm}
In this paper, we investigate the Gauss map of \emph{toric} $X \subset \PN$;
more precisely, we consider a (not necessarily normal)
toric variety $X \subset \PN$ such that the action of the torus on $X$ extends to the whole space $\PN$.
It is known that such $X$ is projectively equivalent to a projective toric variety $X_A$ 
associated to a finite subset $A$ of a free abelian group $M$ (see \cite[Ch.~5, Proposition 1.5]{GKZ}).
The construction of $X_A$ is as follows.

Let $M$ be a free abelian group of rank $n$
and let $\kk [M] = \bigoplus_{u \in M} \kk z^u$ be the group ring of $M$ over $\kk$.
We denote by $T_M$ the algebraic torus $\Spec \kk[M]$.
For a finite subset $A= \{u_0, \ldots, u_N\} \subset M$,
we define the toric variety $X_A$ to be the closure of the image of the morphism
\begin{equation}\label{eq:defphiA}
  \varphi_A : T_M  \arw \PP^N \ : \ t \mapsto [z^{u_0} (t) : \cdots: z^{u_N}(t) ].
\end{equation}
We set $  \langle  A - A  \rangle \subset M $ (resp.\ $ \langle  A - A  \rangle_{\kk} \subset M_{\kk}:= M \otimes_{\Z} \kk $)
to be
the subgroup of $M$ (reap.\ the $\kk$-vector subspace of $ M_{\kk}$) generated by $A -A : = \{ u-u' \, | \, u,u' \in A \} $.
The algebraic torus $T_{ \langle  A - A  \rangle}$ acts on $X_A$,
and $T_{ \langle  A - A  \rangle}$ is contained in $X_A$ as an open dense orbit.
In this paper, for short, we say a projective variety $X \subset \PN$ is \emph{projectively embedded toric}
if $X$ is projectively equivalent to the projective toric variety $X_A$ associated to some finite subset $A$ of a free abelian group $M$ of finite rank.

We denote by $\Aff(A)$ (resp.\ $\Aff_{\kk}(A)$)
the affine sublattice of $M$ (resp.\ the $\kk$-affine subspace of $M_{\kk}$)
spanned by $A$.
In other words,
$\Aff(A)$ (resp.\ $\Aff_{\kk}(A)$) is
the set of linear combinations $\sum_i a_i u_i \in M$ with $a_i \in \ZZ$ (resp.\ $a_i \in \kk$),
$\sum_i a_i = 1$, and $u_i \in A$. We say that $A$ \emph{spans} the affine lattice $M$
(resp.\ the $\kk$-affine space $M_{\kk}$)
if $\Aff(A) = M$ (resp.\ $\Aff_{\kk}(A) = M_{\kk}$).

\vspace{1em}

Projective geometry of toric $X_A$ has been investigated in view of the projective dual
in many papers (\cite{CD}, 
\cite{DDP}, \cite{DN}, \cite{DR}, \cite{GKZ}, \cite{MT}, etc.).
On the other hand, the Gauss map of $X_A$ has not been well studied yet.
In the following result, we describe the structure of Gauss maps of toric varieties
in terms of combinatorics.

\begin{thm}\label{thm_structure}
  Let $\kk$ be an algebraically closed field of arbitrary characteristic, and
  let $M$ be a free abelian group of rank $n$.
  For a finite subset $A= \{u_0, \ldots, u_N\}  \subset M$ which spans the affine lattice $M$,
  set
  \begin{gather*}
    B := \{ u_{i_0} + u_{i_1} + \cdots + u_{i_n} \in M  \, |  \,  u_{i_0}, u_{i_1} , \ldots , u_{i_n}  \text{ span the $\kk$-affine space } M_{\KK} \}
  \end{gather*}
  and let
  $\pi: M \arw M':= M / (\langle B-B \rangle_{\R} \cap M)$ be the natural projection.
  Let $\gamma: X_A \dashrightarrow \G(n, \PP^N)$ be the Gauss map of the toric variety $X_A \subset \PN$.
  Then the following hold.
  \begin{enumerate}  \item\label{thm:item-image}
    The closure $\overline{\gamma(X_A)}$ of the image of $\gamma$,
    which is embedded in a projective space by the Pl\"ucker embedding of $ \Gr (n, \PP^{N})$,
    is projectively equivalent to the toric variety $ X_{B} $.
  \item\label{thm:item-map}
    The restriction of $\gamma : X_A \dashrightarrow \overline{\gamma (X_A)} \cong X_{B}$ on $T_M \subset X_A$
    is the morphism
    \[
    T_M = \Spec \KK[M] \twoheadrightarrow T_{ \langle  B-B  \rangle} = \Spec \KK[ \langle  B-B  \rangle] \subset X_{B}
    \]
    induced by the inclusion $ \langle  B-B \rangle \subset M$.
  \item\label{thm:item-fiber}
    Let $F \subset T_M$ be an irreducible component of any fiber of $\gamma|_{T_M}$ with the reduced structure.
    Let $T_{M'} \hookrightarrow T_M$ be the subtorus induced by $\pi$.
    Then $F$ is a translation of $T_{M'}$ by an element of $T_M$,
    and 
    the closure $\overline{F} \subset X_A$
    is projectively equivalent to the toric variety $X_{\pi(A)}$.

  \end{enumerate}
  In particular,
  we have $ \delta_{\gamma}(X_{A}) = \rk M' = n -\rank  \langle  B-B  \rangle$. 
\end{thm}

Note that there is no loss of generality in assuming that $A $ spans the affine lattice $M$ in the above theorem (see \autoref{thm:span-A-M}).
We also study when the Gauss map of toric $X_A$ is degenerate (i.e., $\rank  \langle  B-B \rangle < n$),
and give a developability criterion for covering families of $X_A$
(see \autoref{sec:criterion-degeneracy}). 

\vspace{2mm}


In the following example, we illustrate the description in \autoref{thm_structure}.

\begin{ex}\label{ex_intro}
  Let $M=\Z^2$ and
  \[
  A = \left\{
    \begin{bmatrix}
      0 \\ 0
    \end{bmatrix},
    \begin{bmatrix}
      0 \\ 1
    \end{bmatrix},
    \begin{bmatrix}
      1 \\ -1
    \end{bmatrix}
    ,\begin{bmatrix}
      -1 \\ -1
    \end{bmatrix}
  \right\} \subset \Z^2,
  \]
  where we write elements in $\Z^2$ by column vectors.
  We consider the Gauss map $\gamma$ of the toric surface $X_A \subset \PP^3$.
  When $\chara \kk \neq 2$,
  \[
  B = \left\{
    \begin{bmatrix}
      0 \\ -1
    \end{bmatrix},
    \begin{bmatrix}
      0 \\ -2
    \end{bmatrix},
    \begin{bmatrix}
      -1 \\ 0
    \end{bmatrix}
    ,\begin{bmatrix}
      1 \\ 0
    \end{bmatrix}
  \right\} \subset \Z^2.
  \]
  Hence $ \langle B-B \rangle =M =\Z^2$, and $\gamma$ is birational due to \ref{thm:item-map} in \autoref{thm_structure}.
  On the other hand, when $\chara \kk =2$,
  \begin{align*}
    B &= \left\{
      \begin{bmatrix}
        -1 \\ 0
      \end{bmatrix}
      ,\begin{bmatrix}
        1 \\ 0
      \end{bmatrix}
    \right\},\
    \lin{B-B} = \left\langle
      \begin{bmatrix}
        2
        \\
        0
      \end{bmatrix}
    \right\rangle,\
    \lin{B-B}_{\RR} \cap M
    = \left\langle
      \begin{bmatrix}
        1 \\ 0
      \end{bmatrix}
    \right\rangle
    \subset \Z^2,
  \end{align*}
  and $\pi(A) = \{0, 1,-1\} \subset \Z^2 / (\langle B-B \rangle_{\R} \cap M) = \Z^1$.
  Thus \ref{thm:item-image} implies that $\overline{\gamma(X_A)} \cong X_{B} =\PP^1$,
  and \ref{thm:item-map} implies that
  $\gamma|_{T_M}: T_M= (\Kx)^2 \twoheadrightarrow T_{\lin{B-B}}=\Kx$ is given by $(z_1,z_2) \mapsto z_1^2$.
  From \ref{thm:item-fiber}, 
  a general fiber of $\gamma$ with the reduced structure is projectively equivalent to the smooth conic $X_{\pi(A)}$.
  \begin{figure}[htbp]
    \[
    \begin{xy}
      (10,0)="A",(0,10)="B",
      (-10,-10)="C",
      (9.9,0.1)="E",(-5,0.1)="F",
      (9.8,-0.1)="I",(-5,-0.1)="J",
      (-20,0)="1",(20,0)="2",
      (0,-17)="3",(0,18)="4",
      (50,-17)="7",(50,18)="8",
      (0,10)*{\bullet},(0,0)*{\bullet},
      (-10,-10)*{\bullet},(10,-10)*{\bullet},
      (-10,0)*{\times},(10,0)*{\times},
      (50,10)*{\sqbullet},(50,0)*{\sqbullet},
      (50,-10)*{\sqbullet},
      (53,10)*{1},(53,0)*{0},(55,-10)*{-1},
      (16,18)*{\bullet\;A}, 
      (16,13)*{\times\;B},
      (65,16)*{\sqbullet\;\pi(A)},
      (35,3)*{\pi}
      \ar "1";"2"
      \ar "3";"4"
      \ar "7";"8"
      \ar (27,0);(43,0)
    \end{xy}
    \]
    \caption{$\chara \kk =2$.}
    \label{figure1}
  \end{figure}
\end{ex}

In characteristic zero, it is well known that, if a projective variety
is the join of some varieties,
then its Gauss map is degenerate due to Terracini's lemma (see \cite[2.2.5]{FP}, \cite[Ch.\,II, 1.10.\,Proposition]{Zak}).
For toric varieties in characteristic zero, this is the only case when the Gauss map is degenerate;
more precisely, we have:

\begin{cor}\label{cor_join}
  Let $X \subset \PP^N$ be a projectively embedded toric variety in $\chara\kk = 0$.
  Then there exist disjoint torus invariant closed subvarieties $ X_0, \ldots,X_{\delta_{\gamma}(X)} \subset X$
  such that $X$ is the join of  $ X_0, \ldots,X_{\delta_{\gamma}(X)}$.
\end{cor}

If the Gauss map of $X_A \subset \PN$ is separable,
$A$ is written as a \emph{Cayley sum} of certain finite subsets $A^0, \dots, A^{\delta_{\gamma}(X)}$ in any characteristic
(see \autoref{thm:subvar-in-join} for details).
However, the statement of \autoref{cor_join} does \emph{not} hold in general in positive characteristic,
even if the Gauss map is separable
(see \autoref{thm:sep-gamma-X-not-join}).

\vspace{1em}

Next, let us focus on inseparable Gauss maps.
A.~H.~Wallace \cite[\textsection 7]{Wallace} showed that
the Gauss map $\gamma$ of a projective variety can be \emph{inseparable}
in positive characteristic.
In this case,
it is possible that
a general fiber of $\gamma$ is \emph{not} a linear subvariety of $\PN$;
the fiber can be a union of points
(H.~Kaji \cite[Example 4.1]{Kaji1986} \cite{Kaji1989}, J.~Rathmann \cite[Example~2.13]{Rathmann}, A.~Noma \cite{Noma2001}),
and can be a non-linear variety
(S.~Fukasawa \cite[\textsection{}7]{Fukasawa2005}).
In fact, Fukasawa \cite{Fukasawa2006}
showed that \emph{any} projective variety appears as a general fiber of the Gauss map of some projective variety.

As we will see in \autoref{cor_rk&deg},
\autoref{thm_structure} provides
several computations on the Gauss map $\gamma$ of toric varieties
(e.g., the rank, separable degree, inseparable degree). 
We also obtain the toric version of Fukasawa's result \cite{Fukasawa2006} as follows:

\begin{thm}[Special case of \autoref{cor_im&fiber}]\label{thm_Fukasawa1}

  Assume $\chara \kk >0$.
  Let $Y \subset \PP^{N'} $ and $Z \subset \PP^{N''}$ be projectively embedded toric varieties.
  If $n:=\dim(Y)+\dim(Z)$ is greater than or equal to $N'$, then
  there exists an $n$-dimensional projectively embedded toric variety
  $X \subset \PP^{n+N''}$
  satisfying the following conditions:
  \begin{enumerate}[\normalfont (i)]
  \item
    (The closure of) a general fiber of the Gauss map $\gamma$ of $X$ with the reduced structure is projectively equivalent to $Y$.
  \item 
    (The closure of) the image of $\gamma$     is projectively equivalent to $Z$.
  \end{enumerate}
\end{thm}

By \autoref{thm_Fukasawa1}, any projectively embedded toric variety appears
as a general fiber and the image of the Gauss map of
a certain projectively embedded \emph{toric} variety; moreover
we can also control the \emph{rank} of $\gamma$,
and the \emph{number} of the irreducible components of a general fiber of $\gamma$
(see \autoref{sec:posit-char-case}, for details).

\vspace{2mm}
This paper is organized as follows.
In \autoref{sec:prel-toric-vari}, we recall some basic properties of toric varieties.
In \autoref{sec:structure-gauss-maps}, we describe the structure of the Gauss maps
of toric varieties in a combinatorial way, and prove \autoref{thm_structure}.
In \autoref{sec:degen-gauss-maps}, we investigate when the Gauss maps are degenerate,
and give a developability criterion.
As a result, we show \autoref{cor_join}.
In \autoref{sec:posit-char-case}, we present two constructions of projectively embedded toric varieties, yielding \autoref{thm_Fukasawa1}.

\subsection*{Acknowledgments}

The authors would like to express their gratitude to Professors
Satoru Fukasawa and Hajime Kaji for their valuable comments and advice.
In particular, \autoref{thm:a-vari-kaji} is due to Professor Kaji.
The first author was partially supported by JSPS KAKENHI Grant Number 25800030.
The second author was partially supported by JSPS KAKENHI Grant Number 25887010.

\section{Preliminary on toric varieties}
\label{sec:prel-toric-vari}

Two projective varieties $X_1 \subset \PP^{N_1}$ and $X_2 \subset \PP^{N_2}$ are said to be
\emph{projectively equivalent} if
there exist linear embeddings $\jmath_i: \PP^{N_i} \hookrightarrow \PN$
(i.e., $\jmath_i^*(\sO_{\PN}(1)) = \sO_{\PP^{N_i}}(1)$)
such that $\jmath_1(X_1) = \jmath_2(X_2)$.
(Indeed, we can take $N = \max\set{N_1,N_2}$.)

\vspace{2mm}
The following two lemmas about toric varieties are well known,
but we prove them for the convenience of the reader.

\begin{lem}\label{lem_lattice_hom}
  Let $\pi : M \arw M'$ be a surjective homomorphism between free abelian groups of finite ranks.
  Let $\iota : T_{M'} \hookrightarrow T_M$ be the embedding induced by $\pi$.
  For a finite set $A \subset M$ with $\Aff(A)=M$
  the closure of $\iota(T_{M'})$ in $X_A$ is projectively equivalent to $X_{\pi(A)}$.
  The translations
  $\set{t \cdot \overline{\iota(T_{M'})}}_{t \in T_M}$
  of the closure $\overline{\iota(T_{M'})}$
  under the action of $T_M$ on $X_A$ give a covering family of $X_A$, and
  each translation is also projectively equivalent to $X_{\pi(A)}$.
\end{lem}

\begin{proof}
  Let $A=\{ u_0,\ldots, u_N\}$ and $\pi(A)= \{u'_0, \ldots, u'_{N'}\}$  for $N=\# A-1$ and $N' = \# \pi(A) -1$.
  We define a linear embedding $\jmath : \PP^{N'} \arw \PP^N$ by
  \[
  \jmath ( [X'_0 : \cdots: X'_{N'}]) = [X_0 : \cdots : X_N],
  \]
  where for each $i$, we set $X_i := X'_j$ for $j$ such that $\pi(u_i)=u'_j$. 
  Then we have the following commutative diagram
  \[
  \xymatrix{
    T_{M'} \ar@{^{(}->}[d]^{\iota} \ar[r]^{\varphi_{\pi(A)}} \ar@{}[dr]|\circlearrowleft &  \PP^{N'} \ar@{^{(}->}[d]^{\jmath} \\
    T_M  \ar[r]^{\varphi_{A}} & \PP^N.\\
  }
  \]
  Hence
  $\overline{\iota(T_{M'})}$ is projectively equivalent to $X_{\pi(A)}$.
  Since the action of $T_M$ on $X_A$ extends to $\PN$ (see \cite[Ch.~5, Proposition~1.5]{GKZ}), translations of $\overline{\iota(T_{M'})}$ are also
  projectively equivalent to $X_{\pi(A)}$.
  Since $\iota(T_{M'})$ is non-empty and contained in $T_M$,
  the translations give a covering family of $X_A$.
\end{proof}

Let $f: X \dashrightarrow Y$ be a rational map between varieties. For a smooth point $x \in X$,
we denote by $d_xf: t_xX \rightarrow t_{f(x)}Y$
the tangent map 
between Zariski tangent spaces at $x$ and $f(x)$.
The \emph{rank} of $f$, denoted by $\rk (f)$, is defined to be the rank of the $\kk$-linear map
$d_xf$ for general $x \in X$.
Recall that $f$ is said to be \emph{separable} if 
the field extension $K(X)/K(f(X))$ is separable;
this condition is equivalent to $\rk(f) = \dim(f(X))$.

\begin{lem}\label{lem_rk&deg}
  Let $M$ be a free abelian group of finite rank.
  Let $M'' $ be a subgroup of $M$ and $g : T_M \twoheadrightarrow T_{M''} $ be the morphism induced by the inclusion $\beta : M'' \hookrightarrow M$.
  \begin{enumerate}[\normalfont \quad (a)]
  \item\label{item:lem_rk&deg:1}
    The inclusions $M'' \subset M''_{\R} \cap M \subset M$ induce
    a decomposition of $g$
    \[
    T_M \stackrel{g_1}{\arw} T_{M''_{\R} \cap M} \stackrel{g_2}{\arw} T_{M''},
    \]
    where
    $g_1$ is a morphism with reduced and irreducible fibers, and $g_2$ is a finite morphism.

  \item\label{item:lem_rk&deg:2} The rank of $g$ is equal to the rank of the $\KK$-linear map $\beta_{\KK} : M''_{\KK} \arw M_{\KK} $ obtained by tensoring $\KK$ with $\beta : M'' \hookrightarrow M$.
    In particular, $g$ is separable if and only if
    $\rk (\beta_{\kk}) = \rank(M'')$.
  \item\label{item:lem_rk&deg:3} Assume $p = \chara \KK > 0$.
    Let $a$ be the index $ [M''_{\R} \cap M : M'']$ and write $a = p^s b$ with integers $s \geq 0, b \geq 1 $ such that $ p \nmid b$.
    Then the degree, separable degree, inseparable degree of the finite morphism $g_2 $ are $a, b , p^s$ respectively.
  \end{enumerate}
\end{lem}

\begin{proof}
  Set $n= \rank M, k= \rank M''$.
  By the elementary divisors theorem (see \cite[III, Theorem~7.8]{Lang}, for example),
  there exists a basis $e_1,\ldots,e_n $ of $M$ such that
  \[
  M'' = \Z a_1 e_1 \oplus \cdots \oplus  \Z a_k e_k \subset   \Z e_1 \oplus \cdots \oplus  \Z e_n =M
  \]
  for some positive integers $a_i$. Set $e''_i := a_i e_i \in M''$.
  By the bases $e_1,\ldots,e_n $ of $M$ and $e''_1,\ldots, e''_k $ of $M''$,
  we identify $M $ and $ M'' $ with $\Z^n$ and $ \Z^k$ respectively.
  Then $g : T_M = (\Kx)^n  \arw T_{M''} = (\Kx)^k $ is described as
  \begin{align}\label{eq_explicit_g}
    (\Kx)^n \arw (\Kx)^k \quad : \quad (z_1, \ldots,z_n) \mapsto (z_1^{a_1}, \ldots,z_k^{a_k}).
  \end{align}
  \vspace{.5mm}

  \noindent
  \ref{item:lem_rk&deg:1} 
  By \ref{eq_explicit_g},
  $g$ is decomposed as
  \[
  (\Kx)^n \arw (\Kx)^k  \arw  (\Kx)^k \quad  : \quad (z_1, \ldots,z_n) \mapsto (z_1, \ldots,z_k) \mapsto (z_1^{a_1}, \ldots,z_k^{a_k}).
  \]
  Since $M''_{\R} \cap M = \Z e_1 \oplus \cdots \oplus  \Z e_k  \subset M$,
  the assertion of \ref{item:lem_rk&deg:1} 
  follows.

  \vspace{1ex}
  \noindent
  \ref{item:lem_rk&deg:3} 
  Write $a_i = p^{s_i} b_i$ by integers $s_i \geq 0, b_i \geq 1 $ such that $ p \nmid b_i$.
  Since $g_2$ is the morphism
  \[
  (\Kx)^k  \arw  (\Kx)^k \quad  : \quad  (z_1, \ldots,z_k) \mapsto (z_1^{a_1}, \ldots,z_k^{a_k}) = (z_1^{p^{s_1} b_1}, \ldots,z_k^{p^{s_k} b_k}),
  \]
  the degree, separable degree, inseparable degree of $g_2 $ are
  $\prod_{i=1}^k a_i =a, \prod_{i=1}^k b_i =b, \prod_{i=1}^k p^{s_i} = p^s$
  respectively.

  \vspace{1ex}
  \noindent
  \ref{item:lem_rk&deg:2} 
  If $\chara \KK=0$, this statement is clear from \ref{eq_explicit_g}.
  Assume $p= \chara \KK >0$ and use the notation $a_i, s_i, b_i$ as above.
  Then the rank of $g$ is equal to
  $\#\set*{1 \leq i \leq k}{s_i = 0}$.
  On the other hand,
  $\beta_{\KK} : M''_{\KK} = \KK^{k} \arw M_{\KK} = \KK^n$ is the $\KK$-linear map defined by $e''_i \mapsto a_i e_i = p^{s_i} b_i e_i  \in M_{\KK}$ for $ 1 \leq i \leq k$.
  Hence the rank of $\beta_{\KK}$ is also equal to $\#\set*{1 \leq i \leq k}{s_i = 0}$.
  Thus we have $\rk (g) =\rk (\beta_{\KK})$.
  The last statement follows from $\dim g(T_M) =\dim T_{M''} =\rk (M'')$.
\end{proof}

\section{Structure of Gauss maps}
\label{sec:structure-gauss-maps}

In this section,
we prove Theorem \ref{thm_structure}
and describe several invariants (e.g., the rank) of Gauss maps of toric varieties
by combinatorial data.

In order to investigate a toric variety $X_A$ for $A \subset M$,
we may assume that $A$ spans the affine lattice $M$ due to the following remark.

\begin{rem}\label{thm:span-A-M}
  For a finite subset $A \subset M$,
  let $\theta : \Z^m \arw  \Aff (A) $ be an affine isomorphism for $ m =\rank \langle A -A \rangle $.
  Then $ \theta^{-1}(A)$ spans $\Z^m$ as an affine lattice and
  $X_{\theta^{-1}(A)}$ is naturally identified with $X_A$
  by \cite[Chapter 5, Proposition 1.2]{GKZ}.
  Hence any projectively embedded toric variety $X$ is projectively equivalent
  to $X_A$ for some $A \subset M$ with $\Aff (A)=M$.
\end{rem}

Let $A = \set{u_0, u_1, \dots, u_N} \subset M:=\ZZ^n$ be a finite subset
which spans the affine lattice $M$.
We denote each $u_i$ by a column vector as
\[
u_i =
\begin{bmatrix}
  u_{i,1} \\ \vdots \\ u_{i,n}
\end{bmatrix}.
\]
Then the morphism $\varphi_A $, defined by \ref{eq:defphiA} in \autoref{sec:introduction},
is described as
\[
\varphi_A: (\Kx)^n \rightarrow \PN \quad :  \quad z = (z_1, \dots, z_n) \mapsto [z^{u_0} : z^{u_1} : \dots : z^{u_N}],
\]
where
$z^{u_i} := z_1^{u_{i,1}} z_2^{u_{i,2}}\cdots z_n^{u_{i,n}}$.
By the assumption that $A$ spans $\ZZ^n$ as an affine lattice,
$\phi_A$ is an isomorphism onto an open subset of $X_A$.

Let us study the Gauss map
$\gamma: X_A \dashrightarrow \Gr(n, \PN)$ of $X_A \subset \PN$.

\begin{lem}\label{thm:matrix-Gamma}
  Let $A, \phi_A$ be as above, and
  let $x \in (\Kx)^n$. Then
  $\gamma(\varphi_A(x)) \in \Gr(n, \PN)$
  is expressed by the $\KK$-valued $(n+1) \times (N+1)$ matrix $\Gamma(x)$;
  more precisely, $\gamma(\varphi_A(x))$
  corresponds to the $n$-plain (i.e., $n$-dimensional linear subvariety of $\PN$)
  spanned by the $n+1$ points which are given as the row vectors of $\Gamma(x)$, where
  \begin{align*}
    \Gamma &: =
    \begin{bmatrix}
      z^{u_0} & z^{u_1} & \cdots & z^{u_N}
      \\
      u_{0,1} \cdot z^{u_0} & u_{1,1} \cdot z^{u_1} & \cdots & u_{N,1} \cdot z^{u_N}
      \\
      \vdots & \vdots && \vdots
      \\
      u_{0,n} \cdot z^{u_0} & u_{1,n} \cdot z^{u_1} & \cdots & u_{N,n} \cdot z^{u_N}
    \end{bmatrix}
    \\
    & \ =
\Bigg[     \,
    z^{u_0} \cdot 
    \begin{bmatrix}
      1
      \\
      u_0
    \end{bmatrix}
    \ 
    z^{u_1} \cdot 
    \begin{bmatrix}
      1
      \\
      u_1
    \end{bmatrix}
    \ 
    \cdots
    \ 
    z^{u_N} \cdot 
    \begin{bmatrix}
      1
      \\
      u_N
    \end{bmatrix}
    \,
  \Bigg]
\end{align*}
\end{lem}
\begin{proof}
  Let $L_x \subset \PN$ be the $n$-plane spanned by
  the $n+1$ points which are given as the row vectors of
  \begin{equation}\label{eq:mat-TTxX-0}
    \begin{bmatrix}
      z^{u_0} & \cdots & z^{u_N}
      \\
      \diff{(z^{u_0})}{z_1} & \cdots & \diff{(z^{u_N})}{z_1}
      \\
      \vdots && \vdots
      \\
      \diff{(z^{u_0})}{z_n} & \cdots & \diff{(z^{u_N})}{z_n}
    \end{bmatrix}(x).
  \end{equation}
  Then $L_x$ coincides with the embedded tangent space $\TT_xX$,
  because of the equality $t_{x'}\hat L_x = t_{x'}\hat X$ of Zariski tangent spaces in $t_{x'}\A^{N+1}$ with $x' \in \hat x \setminus \set{0}$,
  where $\hat S \subset \A^{N+1}$ means the affine cone of $S \subset \PN$.
  On the other hand, \ref{eq:mat-TTxX-0} is calculated as
  \[
  \begin{bmatrix}
    z^{u_0} & \cdots & z^{u_N}
    \\
    u_{0,1} \cdot z^{u_0}/z_1 & \cdots & u_{N,1} \cdot z^{u_N}/z_1
    \\
    \vdots && \vdots
    \\
    u_{0,n} \cdot z^{u_0}/z_n  & \cdots & u_{N,n} \cdot z^{u_N}/z_n
  \end{bmatrix}.
  \]
  Since each row vector means the homogeneous coordinates of a point of $\PN$,
  by multiplying $z_i$ with the $(i+1)$-th row vector for $1 \leq i \leq n$,
  we have the matrix $\Gamma$ and the assertion.
\end{proof}

We interpret \autoref{thm:matrix-Gamma} by the Pl\"ucker embedding.
We regard $\PN$ as $\cP(V) = (V\setminus \{0\}) / \Kx$, the projectivization of $V:=\KK^{N+1}$.
Let $\Gr(n, \PN) \hookrightarrow \cP (\bigwedge^{n+1} V)$ be the Pl\"ucker embedding
and $[p_{i_0,\ldots,i_n}]_{(i_0 , i_1 , \ldots , i_n) \in I}$ be the Pl\"ucker coordinates on $\cP (\bigwedge^{n+1} V)$,
where
\[
I := \set{  ( i_0, i_1, \dots, i_n ) \in \N^{n+1} \, | \, 0 \leq i_0 < i_1 < \cdots < i_n \leq N }.
\]
\begin{lem}\label{lem_plucker}
  The composite morphism  $(\Kx)^n \stackrel{\gamma\circ \varphi_A}{\longrightarrow} \Gr(n, \PN) \hookrightarrow \cP (\bigwedge^{n+1} V)$
  maps $ z = (z_1, \dots, z_n) \in (\Kx)^n $ to    \[
  \left[
    \mu_{i_0, i_1, \dots, i_n}
    \cdot z^{u_{i_0} + u_{i_1} + \cdots + u_{i_n}}
  \right]_{( i_0 ,i_1 , \ldots , i_n) \in I} \in \cP \Big( \bigwedge^{n+1} V \Big),
  \]
  where
  \begin{align*}
    \mu_{i_0, i_1, \dots, i_n} &:=
    \det \begin{bmatrix}
      1 & 1 & \cdots & 1
      \\
      u_{i_0} & u_{i_1} & \cdots & u_{i_n}
    \end{bmatrix}
    \in \KK.
  \end{align*}
\end{lem}
\begin{proof}
  This directly follows from \autoref{thm:matrix-Gamma} and the definition of the Pl\"ucker embedding.
\end{proof}

Set
\[
J := \{  ( i_0, i_1, \dots, i_n ) \in I \, | \,  \mu_{i_0, i_1, \dots, i_n} \neq 0 \}.
\]
By definition, 
$\mu_{i_0, i_1, \dots, i_n} \neq 0$
in $\KK$ if and only if $  u_{i_0} , u_{i_1} , \ldots, u_{i_n}$ span $M_{\KK}$ as an affine space. Hence
the finite set $B$ in \autoref{thm_structure} is described as
\[
B = \{ u_{i_0} + u_{i_1} + \dots + u_{i_n} \in M \, | \, (i_0 , i_1 , \ldots , i_n) \in J \}.
\]
Write $B = \{ b_0,b_1,\cdots,b_{\# B-1} \}$ for mutually distinct $b_j \in M$.
We define a linear embedding
\begin{equation}\label{eq:iota-d-hookrar}
  \begin{aligned}
    \PP^{\# B -1} &\hookrightarrow \cP \Big(\bigwedge^{n+1} V\Big)
    \\
    [Y_0 : Y_1 : \cdots : Y_{\# B -1}] &\mapsto [p_{i_0,i_1,\ldots,i_n}]_{(i_0 , i_1 , \ldots , i_n) \in I}
  \end{aligned}
\end{equation}
as follows. 
When $(i_0, i_1, \dots, i_n) \in J$,
there exists a unique 
$0 \leq j \leq \# B -1$ such that $b_j = u_{i_0} + u_{i_1} + \dots + u_{i_n}$.
For this $j$, we set
$p_{i_0,i_1,\ldots,i_n} = \mu_{i_0, i_1, \dots, i_n} \cdot Y_j$.
When $(i_0, i_1, \dots, i_n) \not\in J$,
we set
$p_{i_0,i_1,\ldots,i_n} = 0$.


By \autoref{lem_plucker} and the definition of the embedding~\ref{eq:iota-d-hookrar},
we have the following commutative diagram:
\begin{equation}\label{eq:pf-thm1}
  \begin{split}
    \xymatrix{
      (\Kx)^n  \ar[rrd]^{\varphi_{B}} \ar@{^{(}->}[r]^(0.6){\varphi_A} & X_A \ar@{-->}[r]^(0.3){\gamma} & \cP (\bigwedge^{n+1} V) \\
      & & \PP^{\# B -1} \ar@{^(->}[u]{}.
    }
  \end{split}
\end{equation}

\begin{proof}[Proof of \autoref{thm_structure}]
  By taking a basis of $M$,
  we may assume that $M = \Z^n$
  and use the notation as above.
  From the above diagram~\ref{eq:pf-thm1}, the embedding $\PP^{\# B -1} \hookrightarrow \cP (\bigwedge^{n+1} V)$
  gives an isomorphism between
  \[
  \overline{\gamma(X_A)} = \overline{\gamma \circ \varphi_A(  (\Kx)^n)}
  \textand X_{B}=\overline{\varphi_{B}((\Kx)^n )}.
  \]
  Hence \ref{thm:item-image} in \autoref{thm_structure} holds.

  The toric variety $X_{B}=\overline{\varphi_{B}((\Kx)^n )}$ contains $T_{\langle B-B \rangle}$ as an open dense orbit
  and
  the restriction $\varphi_B |_{T_M}$ is nothing but
  $T_M= (\Kx)^n \twoheadrightarrow T_{\langle B-B \rangle}$ induced by the inclusion $\langle B-B \rangle \hookrightarrow M$. Hence \ref{thm:item-map} in \autoref{thm_structure} holds by the diagram \ref{eq:pf-thm1}.

  To show \ref{thm:item-fiber}, we use the following claim.

  \begin{claim}\label{thm:gammaTM-stein-fact}
    The morphism $ \gamma|_{T_M} =\varphi_B$ is decomposed as
    \begin{align*}    T_M \stackrel{g_1}{\longrightarrow} T_{\langle B-B \rangle_{\R} \cap M} \stackrel{g_2}{\longrightarrow} T_{\langle B-B \rangle}
    \end{align*}
    by $ \langle B-B \rangle \subset \langle B-B \rangle_{\R} \cap M \subset M$,
    where $g_1$ is a morphism with reduced and irreducible fibers, and $g_2$ is a finite morphism.
  \end{claim}

  \begin{proof}[Proof of \autoref{thm:gammaTM-stein-fact}]
    By applying \ref{item:lem_rk&deg:1} in \autoref{lem_rk&deg} to $ \langle B-B \rangle \subset M$,
    we have the assertion.
  \end{proof}

  The short exact sequence
  \[
  0 \arw  \langle B-B \rangle_{\R} \cap M \arw M \stackrel{\pi}{\arw} M /(\langle B-B \rangle_{\R} \cap M) \arw 0
  \]
  induces a short exact sequence of algebraic tori
  \begin{align}\label{eq_alg_gp}
    1 \arw T_{M /(\langle B-B \rangle_{\R} \cap M)} \arw T_M \stackrel{g_1}{\arw} T_{ \langle B-B \rangle_{\R} \cap M} \arw 1.
  \end{align}
  Hence $g_1^{-1}(1_{T_{\langle B-B \rangle_{\R} \cap M}}) = T_{M /(\langle B-B \rangle_{\R} \cap M)}$ holds for
  the identity element $1_{T_{\langle B-B \rangle_{\R} \cap M}} $ of the torus $T_{\langle B-B \rangle_{\R} \cap M}$.
  Applying \autoref{lem_lattice_hom} to the surjection $\pi:  M \arw M /(\langle B-B \rangle_{\R} \cap M) $,
  it holds that the closure
  \[
  \overline{g_1^{-1}(1_{T_{\langle B-B \rangle_{\R} \cap M}}) }  \subset X_A
  \]
  is projectively equivalent to $ X_{\pi(A)}$.

  Let $F $ be an irreducible component of any fiber of $\gamma |_{T_M}$ with the reduced structure.
  From \autoref{thm:gammaTM-stein-fact},
  $F$ is a fiber of $ g_1$.
  By \ref{eq_alg_gp},
  $F$ is the translation of $g_1^{-1}(1_{T_{\langle B-B \rangle_{\R} \cap M}}) = T_{M /(\langle B-B \rangle_{\R} \cap M)} $
  by an element of $ T_M$.
  Hence the closure $ \overline{F}$ is projectively equivalent to $ X_{\pi(A)}$
  by \autoref{lem_lattice_hom}.
\end{proof}

\begin{cor}\label{cor_rk&deg}
  Let $A, M, \gamma, B$ be as in \autoref{thm_structure}.
  Let $\gamma |_{T_M} = g_2 \circ g_1$ be the decomposition of $\gamma |_{T_M}$ in \autoref{thm:gammaTM-stein-fact}.
  Then the following hold.

  \begin{enumerate}
  \item\label{item:rkdeg-rk} The rank of $\gamma$ is equal to $\dim (\Aff_{\KK} (B))$.
    In particular, $\gamma $ is separable if and only if $ \dim (\Aff_{\KK} (B)) = \rank (\Aff (B))$.

  \item\label{item:rkdeg-deg} Assume $p = \chara\KK > 0$. Then we have
    \[
    \deg(g_2) = [\lin{B-B}_{\RR} \cap M : \lin{B-B}].
    \]
    For the maximum integer $s \geq 0$ such that $p^s \mid \deg(g_2)$,
    the separable degree and the inseparable degree of $g_2 $
    are $\deg(g_2)/p^s $ and $ p^s$ respectively.
    Hence the number of the irreducible components of a general fiber of $\gamma$
    is equal to $\deg(g_2)/p^s$, which is coprime to $p$.

  \end{enumerate}
\end{cor}

We note that $\dim (\Aff_{\KK}(B) )= \dim \lin{B-B}_{\KK}$  holds since $ \Aff_{\KK}(B)$ is a parallel translation of $\lin{B-B}_{\KK} $ in $M_{\KK}$.

\begin{proof}[Proof of \autoref{cor_rk&deg}]
  We apply \autoref{lem_rk&deg} to $ M'' =\langle B-B \rangle \subset M$.
  In this case,
  $g$ in \autoref{lem_rk&deg} is $\gamma |_{T_M}$.
  Since the image of $M''_{\KK} \arw  M_{\KK}$ is nothing but $\langle B-B \rangle_{\KK} \subset M_{\KK}$,
  it holds that $ \rk (\gamma) = \dim \langle B-B \rangle_{\KK} =  \dim \Aff_{\KK} (B)$ by \ref{item:lem_rk&deg:2} in \autoref{lem_rk&deg}. This implies \ref{item:rkdeg-rk}.
  On the other hand,
  \ref{item:rkdeg-deg} follows from \ref{item:lem_rk&deg:3} in \autoref{lem_rk&deg} directly.
\end{proof}

In the following examples,
we denote the separable degree and the inseparable degree of a finite morphism $f$ by $\deg_s(f)$ and $\deg_i(f)$ respectively.

\begin{ex}\label{ex_intro:rkdeg:2}
  Let $A, B$ be as in \autoref{ex_intro} and
  assume $\chara\kk = 2$.
  Then $\rk (\gamma) = \dim \lin{B-B}_{\kk} = 0$.
  From
  \ref{item:rkdeg-deg} of \autoref{cor_rk&deg}, we can calculate $\deg(g_2) = \deg_i(g_2) = 2$ and $\deg_s(g_2) = 1$.
\end{ex}

\begin{ex}\label{thm:kaji's-ex}
  Kaji's example \cite[Example 4.1]{Kaji1986} can be interpreted as follows.

  Let $A = \set{0, 1, cp^m, cp^m+1} \subset M = \ZZ^1$,
  where $c, m$ are positive integers and $p = \chara\kk > 0$.
  Assume that $c$ and $p$ are relatively prime.
  Then
  \[
  B = \set{1, cp^m+1, 2cp^m+1},\
  \lin{B-B} = \lin{cp^m},\
  \lin{B-B}_{\RR} \cap M
  = M.
  \]
  Therefore $\deg(\gamma) = cp^m$, $\deg_i(\gamma) = p^m$, $\deg_s(\gamma) = c$.
  In particular, a general fiber of $\gamma$ with the reduced structure is equal to a set of $c$ points.
  \begin{figure}[htbp]
    \[
    \begin{xy}
      (-15,0)="1",(110,0)="2",
      (0,0)*{\bullet},
      (12,0)*{\bullet},
      (40,0)*{\bullet},
      (52,0)*{\bullet},
      (12,0)*{\mbox{\Large$\times$}},
      (52,0)*{\mbox{\Large$\times$}},
      (92,0)*{\mbox{\Large$\times$}},
      (0,5)*{0},
      (12,5)*{1},
      (40,5)*{cp^m},
      (55,5)*{cp^m+1},
      (92,5)*{2cp^m+1},
      (66,13)*{\bullet\;A},
      (79,13)*{\mbox{\Large$\times$}\,B},
      \ar "1";"2"
    \end{xy}
    \]
    \caption{Kaji's example.}
  \end{figure}
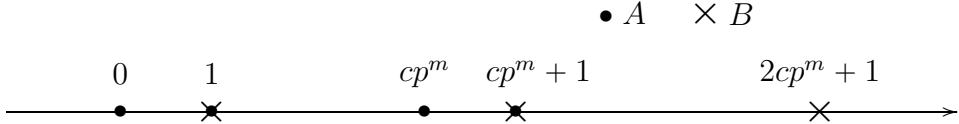

\end{ex}

As in \ref{item:rkdeg-deg}  of \autoref{cor_rk&deg},
the number of the irreducible components of a general fiber of $\gamma$ is coprime to $p = \chara\kk$
in the toric case.
However, the number can be a multiple of $p$ in the non-toric case.
The authors learned the following example from H.~Kaji
in a personal communication.

\begin{rem}[{A variant of \cite[Example 4.1]{Kaji1986}}]
  \label{thm:a-vari-kaji}
  We take $f: \PP^1 \dashrightarrow \PP^1$ to be a separable rational map whose degree is a multiple of $p$,
  and locally parameterize $g$ by
  $t \mapsto [1 : f_1(t)]$.
  We set $\phi: \PP^1 \dashrightarrow \PP^3$ to be 
  the rational map
  which is locally parameterized by
  \[
  t \mapsto [1: t: f_1^p: t \cdot f_1^p]
  \]
  and set $X \subset \PP^3$ to be the projective curve $\overline{\im(\phi)}$.
  Then a general fiber of $\gamma: X \dashrightarrow \Gr(1, \PP^3)$ with the reduced structure
  is equal to a set of $\deg(f)$ points.
  In order to show this, we may check that
  $\gamma$ is locally parameterized by
  $t \mapsto f_1^p(t)$, whose separable degree is equal to $\deg(f)$.
  We leave the details to the reader.

\end{rem}

The following are examples of toric varieties $X_A \subset \PP^N$ with codimension $1$ or $2$ such that the Gauss maps are birational.
Later, these examples will be used
in the proof of \autoref{cor_rk&fiber}.

\begin{ex}\label{prep-cor_rk&fiber}
  Assume $p = \chara \KK >0$.
  Let $e_1, \dots, e_n$ be the standard basis of $M:=\Z^n$, and set
  \[
  A = \{ 0, e_1,\ldots,e_n, a_1 e_1 + \ldots + a_n e_n \}  \subset \Z^n
  \]
  for $a_1 , \ldots, a_n \in \Z$ such that
  \begin{itemize}
  \item[(i)] $a_1, \ldots, a_n \not \equiv 0 \mod p$, and 
  \item[(ii)] $a_1 + \cdots + a_n  \not \equiv 1 \mod p$.
  \end{itemize}
  Then the Gauss map of the toric hypersurface $X_A \subset \PP^{n+1}$ is birational
  as follows.

  By condition (i),
  $A \setminus \{ e_i\}$ spans $M_{\KK}=\KK^n$ as an affine space for any $1 \leq i \leq n$.
  Hence it holds that
  \begin{equation}\label{eq:e1+..:1}
    e_1 + \cdots + e_n + a_1 e_1 + \ldots + a_n e_n  - e_i \in B.
  \end{equation}
  By condition (ii),
  $A \setminus \{ 0\}$ spans $M_{\R}=\KK^n$  as an affine space.
  Hence 
  \begin{equation}\label{eq:e1+..:2}
    e_1 + \cdots + e_n + a_1 e_1 + \ldots + a_n e_n \in B.
  \end{equation}
  Considering the difference between \ref{eq:e1+..:1} and \ref{eq:e1+..:2},
  we have $e_i \in B-B$ for any $i$. Therefore  $\langle B-B \rangle = M$.
  By \ref{thm:item-map} in \autoref{thm_structure},
  the Gauss map of $X_A$ is birational.

  For example,
  the conditions (i) and (ii) are satisfied for $a_1= \cdots = a_n=1$ (resp.\ $a_1= \cdots = a_n= -1$)
  when $n \not \equiv 1 \mod p$ (resp.\ $n \not \equiv  -1 \mod p$).
  Hence there exists a toric hypersurface in $\PP^{n+1}$ whose Gauss map is birational
  if $\chara \kk \neq 2$ or $n$ is even.
\end{ex}

As we will see later in \autoref{rem_rk_in_char2},
the Gauss map of any (not necessarily toric) hypersurface in $\PP^{n+1}$ cannot be birational if $\chara \kk =2$ and $n$ is odd.

\begin{ex}\label{ex_char2}
  Assume $n \geq 2$.
  Let
  \[
  A =\{0, e_1, \ldots, e_n, e_1 +e_2, e_2+e_3+ \cdots + e_n \} \subset M:=\Z^n.
  \]
  For $S := \{ 0, e_1, \ldots, e_n, e_1+e_2\} \subset A$, each of
  \[
  S\, \setminus \{e_1 + e_2\},\ S\, \setminus \{e_1\},\ S\, \setminus \{e_2\}
  \]
  spans the affine space $M_{\KK}$.
  Hence
  it holds that $e_1,e_2 \in  B-B $.
  In addition,
  $\{0, e_1, \ldots, e_n, e_2+e_3+ \cdots + e_n \} \setminus \{e_i\}$ spans the affine space $M_{\KK}$ for $2 \leq i \leq n$.
  Thus we have $e_i -e_2 \in  B-B $ for $2 \leq i \leq n$.
  Hence $\langle B-B \rangle = M$ and the Gauss map of $X_A \subset \PP^{n+2}$ is birational. 
\end{ex}

\section{Degenerate Gauss maps}
\label{sec:degen-gauss-maps}

\subsection{Developability criterion}
\label{sec:criterion-degeneracy}
By Theorem \ref{thm_structure},
the Gauss map of any given toric variety and
its general fibers can be explicitly determined
from the computation of $B$ and $\pi(A)$ as in \autoref{ex_intro}.
However,
it is not so clear for what kind of $A$ the Gauss map $\gamma$ of $X_A \subset \PN$ is degenerate,
i.e.,
$\rank  \langle  B-B \rangle < n$.
The following result tells us
a condition that $\gamma$ is degenerate.

\begin{prop}\label{main_thm}
  Let $M,A,\pi$ be as in Theorem \ref{thm_structure}.
  Let $\tilde\pi : M \arw \tilde{M}'$ be a surjective homomorphism of free abelian groups.
  Then $\tilde\pi$ 
  factors through $\pi: M \rightarrow M'$,
  i.e., $\langle B-B \rangle \subset \ker \tilde \pi$ 
  if and only if
  \[
  \sum_{j=0}^{\tNd} \dim \Aff_{\KK}(A_j) =n- \tNd,
  \]
  where $\tNd := \# \tilde\pi(A) -1$,
  $\tilde\pi(A ) = \{ \tilde u'_0,\ldots, \tilde u'_{\tNd}\}$, and $A_j= \tilde\pi^{-1}(\tilde u'_j) \cap A$.
\end{prop}

For a surjective homomorphism $\tilde\pi: M \rightarrow \tilde M'$
of free abelian groups,
the toric variety
$X_A$ is covered by the translations of
$\overline{T_{\tilde M'}} = X_{\tilde\pi(A)}$ by elements of $T_M$
due to \autoref{lem_lattice_hom}.
The homomorphism $\tilde\pi$ factors through $\pi$
if and only if
$\overline{T_{\tilde M'}}$ (or equivalently,
the translation of $\overline{T_{\tilde M'}}$ by any element of $T_M$)
is  contracted to one point by $\gamma$.

In general, a covering family $\set{F_{\alpha}}$ of a projective variety $X \subset \PN$ by subvarieties $F_{\alpha} \subset X$ is said to be \emph{developable} if
$F_{\alpha}$ is contracted to one point by the Gauss map of $X$
(i.e., $\TT_xX$ is constant on general $x \in F_{\alpha}$) for general $\alpha$.
\autoref{main_thm} is regarded as
a toric version of the developability criterion
(cf.\ \cite[2.2.4]{FP}, \cite{Fukasawa2005}, \cite[\textsection{}4]{expshr}; see also \autoref{sec:separable-gauss-maps} for the separable case).

Before the proof,
we illustrate \autoref{main_thm} by an example.

\begin{ex}\label{ex_intro2}
  Let $A \subset \Z^2$ be as in \autoref{ex_intro}.
  For the projection $\tilde \pi : \Z^2 \arw  \Z^1$ to the second factor,
  $\tilde \pi(A) = \{0,1, -1\}$. Then
  $A_0 =\tilde \pi^{-1}(0) \cap A,\, A_1 =\tilde \pi^{-1}(1) \cap A,\, A_2 =\tilde \pi^{-1}(-1) \cap A$
  are given by
  \[
  A_0 = \left\{
    \begin{bmatrix}
      0 \\ 0
    \end{bmatrix}
  \right\},\,
  A_1 = \left\{
    \begin{bmatrix}
      0 \\ 1
    \end{bmatrix}
  \right\},\,
  A_2 = \left\{
    \begin{bmatrix}
      1 \\ -1
    \end{bmatrix},
    \begin{bmatrix}
      -1 \\ -1
    \end{bmatrix}
  \right\}.
  \]
  Thus
  \[
  \sum_{j=0}^{\tNd} \dim \Aff_{\KK}(A_j) = \left\{
    \begin{array}{cl} 
      1  & \text{when }   \chara \KK \neq2,\\ 0 & \text{when } \chara \KK=2.\\
    \end{array} \right.
  \]
  On the other hand,
  $n-\tNd=2-2=0$ holds.
  Hence the equality in \autoref{main_thm} holds if and only if $\chara \KK =2$.
  Note that, in this example, the above
  $\tilde\pi$ can be identified with the natural projection
  $\pi: M \rightarrow M'$ in \autoref{thm_structure} when $\chara \KK=2$.
\end{ex}

To prove \autoref{main_thm}, we need the following lemma.

\begin{lem}\label{thm:AffAj-eq}
  In the setting of \autoref{main_thm},
  the homomorphism $\tilde \pi$ factors through $\pi$ if and only if
  \[
  \Aff_{\KK}(A_j) = \Aff_{\KK}(\set{u_{i_0}, u_{i_1}, \dots, u_{i_n}} \cap A_j)
  \quad (\RNj)
  \]
  for any $u_{i_0}, \dots, u_{i_n} \in A$ which span the affine space $M_{\KK}$.

\end{lem}

\begin{proof}
  First, we show the ``only if'' part.
  Assume that $\tilde \pi$ factors through $\pi$.
  The inclusion ``$\supset$'' always holds.
  We show ``$\subset$''.
  Let $u \in A_j$.
  Since $u_{i_0}, \dots, u_{i_n}$ span the affine space $M_{\kk}$,
  we can write $u = \sum_{k=0}^n c_{i_k} u_{i_k}$
  with $\sum_{k=0}^n c_{i_k} = 1$ and $c_{i_k} \in \KK$.
  For any $k$ with $c_{i_{k}} \neq 0$, we have $u_{i_{k}} \in A_j$ as follows.

  If $c_{i_{k}} \neq 0$, we find that $\{u\} \cup \set{u_{i_{k'}}}_{0 \leq k' \leq n, k' \neq k}$ span
  the affine space $M_{\KK}$. Thus $u + \sum_{0 \leq k' \leq n, k' \neq k} u_{i_{k'}} \in B$, 
  and then $u - u_{i_{k}} \in B-B$.
  Since $\tilde \pi$ factors through $\pi$, we have $ \tilde\pi(u_{i_{k}}) = \tilde\pi(u) = \tilde{u}'_j$ by $u \in A_j = \tilde{\pi}^{-1}(\tilde{u}'_j) \cap A$;
  hence $u_{i_{k}} \in A_j$.
  
  Thus we have $u = \sum_{u_{i_k} \in A_j} c_{i_k} u_{i_k}$ with $\sum_{u_{i_k} \in A_j} c_{i_k} =1 $,
  i.e.,
  $u $ is contained in $ \Aff_{\KK}(\set{u_{i_0}, u_{i_1}, \dots, u_{i_n}} \cap A_j)$.
  This implies the assertion.

  \vspace{2mm}
  Next, we show the ``if'' part.
  For any $b \in B$,
  we can write $b= u_{i_0} + \cdots + u_{i_n}$ with $u_{i_0} , \ldots , u_{i_n} \in A$ which span the affine space $M_{\KK}$.
  Since the $n$-dimensional affine space $M_{\KK} $ is spanned by $n+1$ elements $u_{i_0} , \ldots , u_{i_n} \in A$,
  we have
  \[
  \# (\set{u_{i_0}, \dots, u_{i_n} }\cap A_j) = \dim  \Aff_{\KK}(\set{u_{i_0}, u_{i_1}, \dots, u_{i_n}} \cap A_j) +1.
  \] 
  Hence $ \# (\set{u_{i_0}, \dots, u_{i_n} }\cap A_j)  = \dim \Aff_{\KK}(A_j) + 1 $ holds for each $0 \leq j \leq \tNd$ by assumption.
  In particular,
  it holds that
  \[
  \tilde \pi(b)= \tilde \pi (u_{i_0} + \cdots + u_{i_n}) = \sum_{0 \leq j \leq \tNd} (\dim \Aff_{\KK}(A_j) + 1) \cdot u'_j ,
  \]
  which does not depend on $b \in B$.
  Thus we have $B-B \subset \ker \tilde \pi$,
  i.e.,
  $\tilde \pi $ factors through $\pi$.
\end{proof}

\begin{rem}\label{thm:B-B-wr-Aj-Aj}
  Assume that $\tilde\pi$ factors through $\pi$.
  From the ``if'' part in the above proof,
  $\# (\set{u_{i_0}, \dots, u_{i_n} }\cap A_j)  $ does not depend on $u_{i_0}, \dots, u_{i_n} \in A$ which span the affine space $M_{\KK}$.
  Thus each element of $B -B$ is written as a linear combination
  of elements of $\bigcup_{\RNj} (A_j-A_j)$.
\end{rem}

\begin{proof}[Proof of \autoref{main_thm}]
  Let us take $u_{i_0}, \dots, u_{i_n} \in A$ which span the affine space $M_{\KK}$.
  Then the latter condition of \autoref{thm:AffAj-eq} holds
  if and only if
  \begin{equation}\label{eq:sub-AffAj}
    \dim \Aff_{\KK}(A_j) = \#(\set{u_{i_0}, u_{i_1}, \dots, u_{i_n}} \cap A_j) - 1
  \end{equation}
  for any $\RNj$ (we note that ``$\geq$'' always holds).
  On the other hand, since $A = A_0 \sqcup A_1 \sqcup \dots \sqcup A_{\tNd}$, we have
  \[
  \sum_{\RNj} (\#(\set{u_{i_0}, u_{i_1}, \dots, u_{i_n}} \cap A_j) -1)
  = \# \set{u_{i_0}, u_{i_1}, \dots, u_{i_n}} - (\tNd + 1) = n-\tNd.
  \]
  Hence $\sum_{\RNj[\tNd]} \dim \Aff_{\KK}(A_j) =n- \tNd$ holds
  if and only if the equality~\ref{eq:sub-AffAj} holds for any $\RNj$.
  Therefore this proposition follows from \autoref{thm:AffAj-eq}.
\end{proof}

\begin{cor}\label{cor_fiber_separable}
  Let $A, \pi: M \rightarrow M'$ be as in \autoref{thm_structure}.
  Then it holds that $\rk (\gamma) \leq n-(\#\pi(A)-1)$. Moreover, if $\gamma$ is separable,
  then we have $\# \pi(A) = \rank M' +1$, which means $ X_{\pi(A)}$ is a linear projective space of dimension $\rank M'$.
\end{cor}

\begin{proof}
  We apply \autoref{main_thm} to the homomorphism $\pi$. Then
  it holds that $\sum_{j} \dim \lin{A_j-A_j}_{\KK} =n- (\#\pi(A)-1)$.
  From \autoref{cor_rk&deg}, we have $\rk (\gamma) = \dim \Aff_{\KK} (B) = \dim \lin{B-B}_{\KK}$.
  By \autoref{thm:B-B-wr-Aj-Aj},
  $\lin{B-B}$ is contained in the space $\lin{\set{A_j-A_j}_j}$.
  Thus $\rk (\gamma) \leq n- (\#\pi(A)-1)$ holds.

  If $\gamma$ is separable, $\rk (\gamma) = \rank \langle B-B \rangle = n - \rank (M')$ holds by \autoref{cor_rk&deg}.
  Hence $\rank(M') \geq \#\pi(A)-1$ holds by the above inequality.
  The converse inequality ``$\leq$'' always holds since $\pi(A)$ spans the affine lattice $M'$.
  Since $\pi(A)$ spans the affine lattice $ M'$,
  the equality $\# \pi(A) = \rank M' +1$ means $ X_{\pi(A)}$ is a linear projective space of dimension $\rank M'$.
\end{proof}

\begin{rem}
  The equality ``$\rk (\gamma) = n-(\#\pi(A)-1)$'' does \emph{not} hold in general.
  For example, set $A=\{0,1,p\}  \subset M = \Z^1$ with $p=\chara \KK >0$.
  Then we have
  \[
  B=\{1, p+1\}, \quad \lin{B-B} = \lin{p} , \quad  \lin{B-B}_{\RR} \cap M =M.
  \]
  Thus $\pi : M=\Z^1 \arw M/(\lin{B-B}_{\RR} \cap M) =\{0 \}$ is the zero map.
  Here we have $\rk (\gamma) = \dim \lin{B-B}_{\kk} = 0$
  and 
  $n-(\#\pi(A) - 1) = 1-(1-1) = 1$.
\end{rem}

\subsection{Separable Gauss maps, Cayley sums, and joins}
\label{sec:separable-gauss-maps}

In this subsection,
we study the case when the Gauss map is separable,
and prove Corollary \ref{cor_join} in the characteristic zero case.

\begin{defn}\label{def_cayley}
  Let $l \leq n$ be non-negative integers. Let $e_1,\ldots,e_l$ be the standard basis of $ \Z^l$.
  For finite sets $A^0,\ldots,A^l \subset \Z^{n-l}$,
  the Cayley sum $A^0 * \cdots * A^l  $ of $A^0,\ldots,A^l$ is defined to be
  \[
  A^0 * \cdots * A^l  := (A^0 \times \{0\}) \cup (A^1 \times \{e_1 \}) \cup \cdots (A^l \times \{e_l\})  \subset \Z^{n-l} \times \Z^l.
  \]
\end{defn}

Let $A$ be the Cayley sum of $A^0,\ldots,A^l \subset \Z^{n-l}$,
and assume that $A$ spans the affine lattice $ \Z^{n-l} \times \Z^l$.
For the projection $\tilde{\pi} : \Z^{n-l} \times \Z^l \arw \Z^l $
to the second factor, $X_{\tilde{\pi}(A)}$ is an $l$-plane
since $\tilde\pi(A) = \set{0, e_1, \dots, e_l}$.
By \autoref{lem_lattice_hom}, $X_A$ is covered by $l$-planes,
which are translations of $X_{\tilde{\pi}(A)} = \overline{T_{\Z^l}}$.
For this $\tilde\pi$,
$\tilde N'$ in \autoref{main_thm} is equal to $l$.
Thus,
the subtorus $ T_{\Z^l} \subset T_{\Z^{n-l} \times \Z^l}$ is contracted to one point by the Gauss map of $X_A$
if and only if
\begin{align}\label{eq_condition_for_sum}
  \sum_{j=0}^{l} \dim \Aff_{\KK}(A^j) =n- l.
\end{align}
In other words,
\ref{eq_condition_for_sum} is the condition for the developability
of the covering family obtained by translations of $\overline{T_{\Z^l}}$.
In fact,
any toric variety with separable Gauss map is described by a Cayley sum with the condition~\ref{eq_condition_for_sum}, as follows.

\begin{thm}\label{thm:subvar-in-join}
  Let $A, M$ be as in \autoref{thm_structure}.
  Assume that the Gauss map $\gamma$ of $X_A \subset \PN$ is separable, and set $l = \delta_{\gamma}(X)$.
  Then there exist finite subsets $A^0, \dots, A^{l} \subset \ZZ^{n-l}$
  with $\sum_{j=0}^{l} \dim \Aff_{\KK}(A^j) =n- l$
  such that $A$ is identified with the Cayley sum of $A^0, \dots, A^{l} \subset \ZZ^{n-l}$ under some affine isomorphism $M \simeq \ZZ^{n-l} \times \ZZ^{l}$.
\end{thm}

\begin{proof}
  Let $A = \set{u_0, u_1, \dots, u_N}$ and $\pi : M \arw M'$ be
  as in \autoref{thm_structure},
  where we may assume $u_0=0 \in M$.
  From \autoref{thm_structure} and \autoref{cor_fiber_separable}, it follows that
  $\rank M' = l$ and $\# \pi (A) = l+1$ since $\gamma$ is separable by assumption.
  Set $\pi (A)=\{u'_0, \ldots,u'_l \} $.
  Without loss of generality,
  we may assume that $u'_0 =\pi(u_0)=0 \in M'$.
  Then $ u'_1, \ldots,u'_l $ form a basis of $M'$
  since $\pi(A) $ spans the affine lattice $M' \simeq \Z^l$.
  Set $A_j = \pi^{-1}(u'_j) \cap A$.

  Fix a splitting $s : M' \arw M$ of the short exact sequence $0 \arw \ker \pi \arw M \stackrel{\pi}{\arw} M' \arw 0$.
  Then the induced isomorphism
  \[
  M \stackrel{\sim}{\arw} \ker \pi \times M' \quad : \quad u \mapsto (u - s ( \pi(u)), \pi(u))
  \]
  gives an identification of $A \subset M$ with
  \begin{equation}\label{eq:expressA-Cayley}
    \bigcup_{j=0}^l \left( A^j  \times  \{u'_j \}  \right) \subset \ker \pi \times M' ,
  \end{equation}
  where $A^j := A_j -s(u'_j) \subset \ker \pi$ is the parallel translation of $A_j$ by $s(u'_j)$.
  Since
  $ u'_1, \ldots,u'_l $ form a basis of $M'$, $u_0' = 0$,
  and $\ker \pi \simeq \ZZ^{n-l}$,
  this theorem follows.
\end{proof}

In order to prove \autoref{cor_join}, we consider a relation between
Cayley sums and joins.
For projective varieties $X_1, \dots, X_m \subset \PN$,
we define the \emph{join} of $X_1, \dots, X_m$ to be the closure of
$\bigcup_{x_1 \in X_1, \dots, x_m \in X_m} \Lambda_{x_1, \dots, x_m} \subset \PN$,
where $\Lambda_{x_1, \dots, x_m}$ is the linear variety spanned by the points $x_1, \dots, x_m$.

\begin{lem}\label{lem_torus_inv_sub}
  Let $A \subset \ZZ^{n-l} \times \ZZ^l$ be the Cayley sum of $A^0, \dots, A^l \subset \ZZ^{n-l}$ with $\Aff(A) = \ZZ^{n-l} \times \ZZ^l$.
  Then the following hold.
  \begin{enumerate}[\quad \normalfont (a)]
  \item \label{lem_torus_inv_sub:1}
    $X_{A^0}, \ldots,X_{A^l}$ are embedded into $X_A$ as torus invariant subvarieties,
    and they are mutually disjoint.

  \item \label{lem_torus_inv_sub:2}
    $X_{A}$ is contained in the join of $X_{A^0}, \ldots,X_{A^l} \subset \PN$,
    and the codimension of $X_{A}$ in the join is
    $l - n + \sum_{j=0}^l \rank \Aff(A^j)$.

  \end{enumerate}
\end{lem}

\begin{proof}
  \begin{inparaenum}[(a)]
  \item 
    Write $A^j = \{u^j_{0}, \ldots,u^j_{N_j}\} \subset \Z^{n-l}$ for $ N_j= \# A^j -1$.
    Set $N=\# A -1 = \sum_{j=0}^l (N_j+1) -1$
    and let $\{ X^j_{i}\}_{0 \leq i \leq N_j, 0 \leq j \leq l}$ be the homogeneous coordinates on $ \PP^{N}$.
    By the definition of the Cayley sum $A$,
    it holds that
    \begin{equation}\label{eq:phiAyz}
      \begin{aligned}
        \varphi_{A }(z, w) &= [w_j z^{u_i^{j}}]_{0 \leq i \leq N_j, 0 \leq j \leq l} \\
        &= [ w_0z^{u^0_{0}} : \cdots : w_0 z^{u^0_{N_0}} :
        w_1 z^{u^1_{0}} : \cdots : w_1 z^{u^1_{N_1}} :
        \cdots: w_l z^{u^l_{0}} : \cdots : w_l z^{u^l_{N_l}}]
      \end{aligned}
    \end{equation}
    in $\PP^N$
    for $(z, w)=(z_1,\ldots,z_{n-l}, w_1,\ldots,w_l)
    \in (\Kx)^{n-l} \times (\Kx)^{l} = T_{\Z^{n-l} \times \Z^{l}}$,
    where we set $w_0 = 1$.
    For fixed $0 \leq j \leq l$ and $z \in (\Kx)^{n-l} = T_{\Z^{n-l}}$,
    $\phi_A(z, w)$ converges to
    \begin{equation}\label{eq:phiAjz}
      [0:\cdots : 0 : z^{u^j_{0}} : \cdots : z^{u^j_{N_j}} : 0:\cdots : 0] \in \PP^N
    \end{equation}
    when $w_k/w_j \rightarrow 0$ for $0 \leq k \neq j \leq l$.
    Thus, the point \ref{eq:phiAjz} is
    contained in the closure $\overline{\varphi_{A }(T_{\Z^l \times \Z^{n-l}} )}=X_A$.
    In other words, $\varphi_{A^j}(z) = [z^{u^j_{0}} : \cdots : z^{u^j_{N_j}} ] \in \PP^{N^j} $ is contained in $X_A$ for any $z \in (\Kx)^{n-l} = T_{\Z^{n-l}}$
    by embedding $\PP^{N_j}$ into $\PP^N$ as
    \begin{equation}\label{eq:PNj=Xji}
      \PP^{N_j}  = (X^{j'}_{i}=0)_{j' \neq j, 0 \leq i \leq N_{j'}} \subset  \PP^{N} .
    \end{equation}
    Since $X_{A^j}$ is the closure of $\varphi_{A^j}(T_{\Z^{n-l}})$,
    $X_{A^j} \subset \PP^{N_j}$ is contained in $X_A$.

    The action of $T_{\Z^{n-l} \times \Z^{l}}  $ on $X_A$ is described as
    \[
    (z,w) \cdot [X^j_i]_{0 \leq i \leq N_j, 0 \leq j \leq l} = [ w_j z^{u^j_i} X^j_i]_{0 \leq i \leq N_j, 0 \leq j \leq l}
    \]
    for $ (z,w ) = (z_1,\ldots,z_{n-l}, w_1,\ldots,w_l) \in  T_{\Z^{n-l} \times \Z^{l}} $ and $ [X^j_i]_{0 \leq i \leq N_j, 0 \leq j \leq l} \in X_A$,
    where $w_0 = 1$ as before.
    Therefore $X_{A^j} \subset X_A$ is a torus invariant subvariety.
    Since $\PP^{N_0}, \ldots, \PP^{N_l} \subset \PP^N$ are mutually disjoint by \ref{eq:PNj=Xji},
    so are $X_{A^0}, \ldots, X_{A^l} \subset X_A$.

    \vspace{1ex}
    \noindent
  \item
    For $(z,w ) \in  T_{\Z^{n-l} \times \Z^{l}}$,
    the image $\varphi_{A}(z,w) \in X_A \subset \PN$ is described by \ref{eq:phiAyz}, and
    $\varphi_{A^j}(z) \in X_{A^j} \subset \PN$ is described by \ref{eq:phiAjz}
    for each $0 \leq j \leq l$.
    Hence $\varphi_{A}(z,w)  $ is contained in
    the $l$-plane spanned by $\varphi_{A^0}(z), \varphi_{A^1}(z), \cdots, \varphi_{A^l}(z) $.
    Thus $X_{A}$ is contained in the join of $X_{A^0}, \ldots,X_{A^l}$.
    From \ref{eq:PNj=Xji},
    the dimension of the join of $X_{A^0}, \ldots,X_{A^l}$ is
    $l+ \sum_{j=0}^l \dim X_{A^j}$,
    which is equal to $l + \sum_{j=0}^l \rank \Aff(A^j)$.
    Since $\dim(X_A) = n$, the assertion about the codimension follows.
  \end{inparaenum}
\end{proof}

\begin{ex}
  Let $A \subset \Z^2 \times \Z^1$ be the Cayley sum of
  \[
  A^0 = \Set{\begin{bmatrix}
      0 \\ 0
    \end{bmatrix},
    \begin{bmatrix}
      1 \\ 0
    \end{bmatrix},
    \begin{bmatrix}
      2 \\ 0
    \end{bmatrix}
  }, \
  A^1 = \Set{\begin{bmatrix}
      0 \\ 0
    \end{bmatrix},
    \begin{bmatrix}
      0 \\ 1
    \end{bmatrix},
    \begin{bmatrix}
      0 \\ 2
    \end{bmatrix},
    \begin{bmatrix}
      0 \\ 3
    \end{bmatrix}
  }
  \subset \Z^2.
  \] 
  Then it holds that
  \[
  l - n + \sum_{j=0}^l \rank \Aff(A^j) = 1 - 3 + (1+1) =0.
  \]  
  Hence $X_A$ is the join of $X_{A^0}$ and $X_{A^1}$.
  In fact,
  \[
  X_A = \overline{ \big\{ [1: x: x^2 : w : w y : w y^2 : w y^3]  \, | \,  (x,y,w) \in (\Kx)^3= T_{\Z^2 \times \Z^1} \big\} } \subset \PP^6,
  \]
  and the conic $X_{A^0} \subset \PP^2$ and the twisted cubic $X_{A^1} \subset \PP^3$ are embedded into $X_A$ as
  \begin{align*}
    X_{A^0} &=\overline{ \big\{ [1: x: x^2 : 0 : 0 : 0 : 0] \, | \, x \in \Kx \big\} } \hspace{1mm} \subset X_A, \\
    X_{A^1} &=\overline{ \big\{  [0:0:0 : 1 : y : y^2 : y^3] \, | \, y \in \Kx \big\} }  \subset X_A.
  \end{align*}
\end{ex}

\begin{rem}\label{thm:subvar-in-join:rem}
  In \autoref{thm:subvar-in-join},
  the codimension of $X_A$ in the join of $X_{A^0}, \dots, X_{A^l}$
  is
  $\sum_{j=0}^l (\rank \Aff(A^j) - \dim \Aff_{\kk}(A^j))$
  by \autoref{lem_torus_inv_sub}.
\end{rem}

\vspace{1mm}
Now we can prove \autoref{cor_join} immediately.

\begin{proof}[Proof of Corollary \ref{cor_join}]
  We may assume that
  $X=X_A$ for some finite set $A \subset M$ with $\Aff (A)=M$.
  Then the assertion follows from
  \autoref{thm:subvar-in-join} and \autoref{thm:subvar-in-join:rem} since
  the equality $\rank \Aff(A^j) = \dim \Aff_{\kk}(A^j)$ holds in $\chara\kk = 0$.
\end{proof}

\begin{cor}
  Assume $\chara\kk=0$.
  If a toric variety $X_A \subset \PN$ is the join of some projective varieties,
  then
  $X_A$ is the join of some toric varieties $X_{A^0}, X_{A^1}, \dots, X_{A^{l}}$ for some $l >0$.
\end{cor}
\begin{proof}
  Since $X_A$ is the join in $\chara\kk=0$,
  the Gauss defect $\delta_{\gamma}(X_A)$ is positive (due to Terracini's lemma). Hence this corollary follows from \autoref{cor_join}
  for $l = \delta_{\gamma}(X)$.
\end{proof}

The assumption $\chara\kk = 0$ is crucial in the above proof of \autoref{cor_join}.
In positive characteristic,
even if the Gauss map $\gamma$ of toric $X_A$ is separable (equivalently, 
a general fiber of $\gamma$ is scheme-theoretically an open subset of a linear variety of $\PN$),
it is possible that $\gamma$ is degenerate but $X_A$ is \emph{not} the join of any varieties, as follows.

\begin{ex}\label{thm:sep-gamma-X-not-join}
  Let $p = \chara\kk \geq 3$.
  Set
  \[
  A^0=\{0,1,-1\}, A^1 =\{ 0,p\} \subset \Z^1
  \]
  and let $A \subset \Z^1 \times \Z^1$ be the Cayley sum of $A^0,A^1$.
  Then
  \[
  \lin{B-B} = \Z^1 \times \{0\} \subset \Z^1 \times \Z^1.
  \]
  Hence $\pi: M \rightarrow M/ (\lin{B-B}_{\RR} \cap M)$ coincides with
  the projection $\Z^1 \times \Z^1 \rightarrow \ZZ^1$ to the second factor.
  In this setting, the following hold.

  \begin{enumerate}
  \item The Gauss map $\gamma$ of the surface $X_A \subset \PP^4$ is separable.
    A general fiber of $\gamma$ is a line; in particular,
    $\gamma$ is degenerate.

  \item The conic $X_{A^0}$ and the line $X_{A^1}$ are embedded into $X_A$.

  \item $X_A$ is of codimension one in the join of $X_{A^0}$ and $X_{A^1}$.
    On the other hand, $X_A$ itself is not the join of any varieties.
  \end{enumerate}

  The reason is as follows.
  \begin{inparaenum}
  \item The separability of $\gamma$ follows from \autoref{cor_rk&deg}.
    A general fiber of $\gamma$ is projectively equivalent to $X_{\pi(A)}$, which is a line.
  \item The embedding of $X_{A^j}$ is given as in 
    \autoref{lem_torus_inv_sub}.
  \item It follows from \autoref{thm:subvar-in-join} that
    $X_A$ is contained in the join of $X_{A^0}$ and $X_{A^1}$.
    Since the join is of dimension $3$,
    the codimension of $X_A$ in the join is equal to $1$. By \autoref{thm:subvar-in-join:rem}, the codimension is also calculated from
    \begin{gather*}
      \rk \Aff (A^0) - \dim \Aff_{\kk} (A^0) = 1-1 = 0,
      \\
      \rk \Aff (A^1) - \dim \Aff_{\kk} (A^1) = 1-0 = 1.
    \end{gather*}
    On the other hand, $X_A$ is not the join of any varieties;
    this is because,
    a projective surface $X \subset \PN$ is the join of some varieties
    if and only if $X$ is the cone of a curve with a vertex.
  \end{inparaenum}

 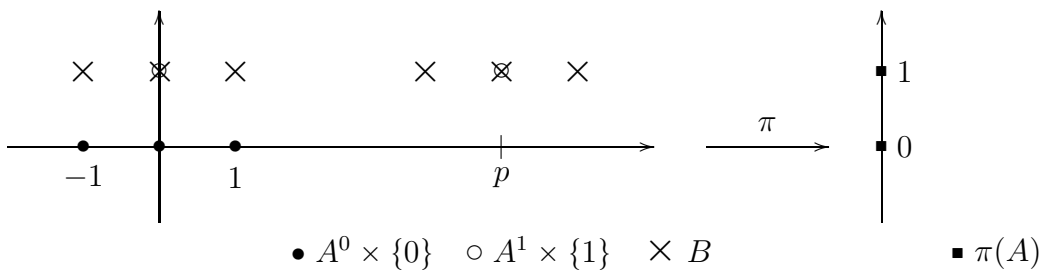
\begin{figure}[htbp]
    \[
    \begin{xy}
      (10,0)="A",(0,10)="B",
      (-10,-10)="C",
      (9.9,0.1)="E",(-5,0.1)="F",
      (9.8,-0.1)="I",(-5,-0.1)="J",
      (-20,0)="1",(65,0)="2",
      (0,-10)="3",(0,18)="4",
      (95,-10)="7",(95,18)="8",
      (10,-4)*{1},
      (-10,-4)*{-1},
      (45,0)*{\mbox{\scriptsize$\mid$}},
      (45,-4)*{p},
      (0.1,10)*{\mbox{\Large$\times$}},
      (-10,10)*{\mbox{\Large$\times$}},
      (10,10)*{\mbox{\Large$\times$}},
      (35,10)*{\mbox{\Large$\times$}},
      (45,10)*{\mbox{\Large$\times$}},
      (55,10)*{\mbox{\Large$\times$}},
      (45,10)*{\mbox{\large$\circ$}},
      (0,10)*{\mbox{\large$\circ$}},
      (0,0)*{\bullet},
      (-10,0)*{\bullet},(10,0)*{\bullet},
      (95,10)*{\sqbullet},(95,0)*{\sqbullet},
      (98,10)*{1},(98,0)*{0},
      (45,-14)*{
        \bullet\;A^0 \times \set{0}\quad 
        \mbox{\large$\circ$}\;A^1 \times \set{1}\quad 
        \mbox{\Large$\times$}\;B
      }, 
      (110,-14)*{\sqbullet\;\pi(A)},
      (80,3)*{\pi}
      \ar "1";"2"
      \ar "3";"4"
      \ar "7";"8"
      \ar (72,0);(88,0)
    \end{xy}
    \]
    \caption{Cayley sum: $A = A^0 * A^1$ \ ($p \geq 3$).}
    \label{figure3}
  \end{figure}

\end{ex}

\section{Constructions in positive characteristic}
\label{sec:posit-char-case}

This section presents two constructions of projectively embedded toric varieties
in positive characteristic.
We consider whether it is possible to find a toric variety
whose Gauss map $\gamma$ has given data about
\begin{enumerate}
\item[{(F)}]   each irreducible component of a general fiber of $\gamma$;
\item[{(I)}]   the image of $\gamma$;
\item[{(c)}]
  the number of the irreducible components of a general fiber of $\gamma$;
\item[{(r)}]
  the rank of $\gamma$.
\end{enumerate}
The statement of \autoref{thm_Fukasawa1}
means that, a projectively embedded toric variety $X$ is constructed for given (F) and (I).
In fact,
in the construction of \autoref{cor_im&fiber}, we can control (F), (I), and (c),
but not (r) (indeed, $\rk (\gamma) = 0$ for $X$ in our proof).
On the other hand, in the construction of \autoref{cor_rk&fiber},
we can control (F), (r), and (c).
Hereafter we assume that $p= \chara \KK$ is positive.

\begin{thm}\label{cor_im&fiber}
  Assume $p = \chara \KK >0$.
  Let $A'$ and $A''$ be finite subsets of free abelian groups $M'$ and $M''$
  respectively
  such that $\Aff (A') = M'$ and $ \Aff (A'')=M''$.
  Let $c > 0$ be an integer coprime to $p$.
  Assume $n:= \rk(M') + \rk(M'') \geq \#A'-1$ and $\rk(M'') \geq 1$.
  Then there exists a finite subset $A \subset M := \Z^n$ with $\# A = n + \# A''$ and $\Aff (A)= M$
  such that the Gauss map $\gamma$ of $X_A \subset \PP^{\#A-1}$ satisfies the following conditions:
  \begin{enumerate}[\normalfont(i)]
  \item (The closure of) each irreducible component of a general fiber of $\gamma$
    is projectively equivalent to $X_{A'}$.
  \item (The closure of) the image $\gamma (X_A)$ is projectively equivalent to $X_{A''}$.
  \item The number of the irreducible components of a general fiber of $\gamma$ is equal to $c$.
  \end{enumerate}
\end{thm}

\newcommand{\sAdm}{N'}\newcommand{\sAddm}{N''}

\begin{proof}We set $\sAdm = \# A'-1$
  and $A'=\{ u'_0,\ldots, u'_{\sAdm}\}$,
  and let $e_1, \ldots, e_{n}$
  be the standard basis of $M = \Z^{n}$.
  Without loss of generality,
  we may assume that $u'_0=0$.
  We define a group homomorphism $\pi$ by
  \begin{align*}
    \pi :  M \arw M' \quad: \quad   e_i \mapsto \left\{ 
      \begin{array}{cl} 
        u'_i & \text{for } 1 \leq i \leq \sAdm ,\\ 
        0 & \text{for }  \sAdm +1 \leq i \leq n .\\
      \end{array} \right.
  \end{align*}
  We note that $n \geq \sAdm$ holds by assumption.
  Since $\Aff (A') = M'$ and $u'_0=0 \in M'$,
  $\pi$ is surjective.
  Hence $\ker \pi$ is a free abelian group whose rank is $\rk (M'')$.
  Since $\rk(M'') \geq 1$, we can take and
  fix an injective group homomorphism
  \[
  M'' \hookrightarrow \ker \pi
  \]
  whose cokernel is isomorphic to $\Z / \lin{c}$. 
  Let $A''=\{ f_0,\ldots, f_{\sAddm}\} \subset M'' \subset \ker \pi$ for $\sAddm = \# A'' -1$.
  Without loss of generality,
  we may assume that $f_0=0 \in M''$.
  Set
  \[
  A =\{ e_1,\ldots, e_{n}, p f_0, \ldots, p f_{\sAddm} \}  \subset M.
  \]
  Since $e_1, \ldots, e_{n} $ is a basis of $M$ and $ p f_0 =0 \in M$,
  $A$ spans the affine lattice $M$.

  Let $B$ be as in the statement of Theorem \ref{thm_structure} for the above $A$.
  Choose $n+1$ elements $u_{i_0}, u_{i_1} , \ldots , u_{i_{n}} \in A$ which span the affine space $M_{\KK}$.
  Since $p f_s=0 $ in $M_{\KK} $ for $0 \leq s \leq \sAddm$,
  at most one element of $\{p f_0, p f_1,\ldots, p f_{\sAddm}  \}$ is contained in $\{u_{i_0}, u_{i_1} , \ldots , u_{i_{n}}\}$.
  Hence $\{u_{i_0}, u_{i_1} , \ldots , u_{i_{n}}\} =\{ p f_s, e_1,\ldots, e_{n} \}$ holds for some $0 \leq s \leq \sAddm $.
  Thus we have
  \[
  B = \{ p f_0 + e_1+ \cdots +e_{n} ,\ldots, p f_{\sAddm} +e_1+ \cdots +e_{n}  \},
  \]
  that is, $B $ is the parallel translation of $p \cdot A'' $ by $e_1+ \cdots +e_{n}$.
  Hence we have $X_B = X_{p \cdot A''} = X_{A''}$.
  Since the closure of the image $\gamma(X_A)$ is projectively equivalent to $X_B$ by \ref{thm:item-image} in Theorem \ref{thm_structure},
  the condition~(ii) in this theorem holds.

  Since $A''= \{f_0, \ldots, f_{\sAddm} \}$ spans the affine lattice $M''$,
  it holds that
  \[
  \langle B -B \rangle = p \cdot \langle A''-A'' \rangle =p \cdot M'' \subset M'' \subset \ker \pi.
  \]
  Therefore $\langle B-B \rangle_{\R} \cap M= \ker \pi$
  and the natural projection $M \arw M / (\langle B-B \rangle_{\R} \cap M)$ coincides with $\pi : M \arw M'$.
  Since $\pi(A) =A'$ by the definition of $\pi$ and $A$,
  the condition~(i) in this theorem follows from \ref{thm:item-fiber} in Theorem \ref{thm_structure}.

  Since $ \ker \pi / M'' \simeq \Z / \lin{c}$,
  the order of the finite group
  \[
  (\langle B-B \rangle_{\R} \cap \Z^{n}) / \langle B-B \rangle = \ker \pi / (p \cdot M'')
  \]
  is $p^{\rk(M'')} c$.
  Hence (iii) in this theorem follows from (2) in \autoref{cor_rk&deg}.
\end{proof}

Note that, in the above construction, $\rk (\gamma) = \dim\lin{B-B}_{\kk} = 0$.

\begin{ex}

  In this example,
  we illustrate \autoref{cor_im&fiber} for 
  \[
  A' = \Set{
    \begin{bmatrix}
      0 \\ 0
    \end{bmatrix},
    \begin{bmatrix}
      1 \\ 0
    \end{bmatrix},
    \begin{bmatrix}
      0 \\ 1
    \end{bmatrix},
    \begin{bmatrix}
      1 \\ 1
    \end{bmatrix}
  } \subset \ZZ^2,\
  A'' = \set{0, 1, 2, 3} \subset \ZZ^1 
  \]
  and an integer $c >0$ coprime to $p = \chara \KK$.
  In this case, $X_{A'} = \PP^1 \times \PP^1 \subset \PP^3$ is a smooth quadric surface
  and $X_{A''} \subset \PP^3$ is a twisted cubic curve. 
  Since $n= 2 +1 \geq \# A' -1 =4-1$,
  we can apply \autoref{cor_im&fiber}.

  We use the notations in the proof of \autoref{cor_im&fiber}.
  Since $\pi : M = \Z^3 \arw M' = \Z^2$ is defined by
  \[
  \pi(e_1)=
  \begin{bmatrix}
    1 \\ 0
  \end{bmatrix}, \quad
  \pi(e_2)=
  \begin{bmatrix}
    0 \\ 1
  \end{bmatrix}, \quad
  \pi(e_3)=
  \begin{bmatrix}
    1 \\ 1
  \end{bmatrix}
  \]
  for the standard basis $e_1,e_2,e_3$ of $\Z^3$,
  $\ker \pi $ is generated by $e_1 +e_2 -e_3$.
  Hence an injection $M'' =\Z^1 \hookrightarrow \ker \pi$ with cokernel $\Z / \lin{c}$ is given by mapping $1 \in \Z^1$ to $ c (e_1 +e_2 -e_3)$.
  Thus
  $A$ in the proof of \autoref{cor_im&fiber} is 
  \[
  A = \Set{
    \begin{bmatrix}
      1 \\ 0 \\ 0
    \end{bmatrix},
    \begin{bmatrix}
      0 \\ 1 \\ 0
    \end{bmatrix},
    \begin{bmatrix}
      0 \\ 0 \\ 1
    \end{bmatrix},
    p \cdot 0 ,
    p \cdot  f,
    p \cdot  2  f,
    p \cdot 3  f
  }
  \text{ for }
  f= \begin{bmatrix}
    c \\ c \\ -c
  \end{bmatrix}.
  \]

  \vspace{2mm}
  We can see directly that (i) - (iii) in \autoref{cor_im&fiber} hold for this $A$ as follows:
  In this case,
  $X_A$ is the image of 
  $\phi_A: (\Kx)^3 \hookrightarrow \PP^{6}$ defined by 
  \begin{align}\label{eq_varphi}
    (x,y,z) \mapsto [x : y : z : 1 : (xyz^{-1})^{pc} : (xyz^{-1})^{2pc} : (xyz^{-1})^{3pc}].
  \end{align}
  We embed $\PP^3 $ into $ \Gr (3,\PP^6)$ by mapping $[X:Y:Z:W] \in \PP^3$ to the $3$-plane in $\PP^6$ spanned by the $4$ points which are given as the row vectors of
  \begin{align*}
    \begin{bmatrix}
      0 & 0 & 0 & X & Y & Z & W 
      \\
      1 & 0 & 0 & 0 & 0 & 0 & 0 
      \\
      0 & 1 & 0 & 0 & 0 & 0 & 0 
      \\
      0 & 0 & 1 & 0 & 0 & 0 & 0 
    \end{bmatrix}.
  \end{align*}
  Then the image of $\varphi_A(x,y,z)$  by the Gauss map $\gamma$ of $X_A$ is
  \begin{align}\label{eq_image}
    [1 : (xyz^{-1})^{pc} : (xyz^{-1})^{2pc} : (xyz^{-1})^{3pc}] \in \PP^3 \subset \Gr(3, \PP^6)
  \end{align}
  from \autoref{thm:matrix-Gamma}.
  Hence the closure $\overline{\gamma (X_A)}$ is the twisted cubic curve $X_{A''} \subset \PP^3$.
  Thus (ii) holds.

  From \ref{eq_varphi} and \ref{eq_image},
  the fiber of $\gamma |_{T_M}$ over $[1:1:1:1] \in \overline{\gamma(X_A)}= X_{A''} \subset \PP^3$ is
  \[
  \left\{ [x : y : z : 1 : 1 : 1 : 1] \in X_A \subset \PP^6 \, | \, (x,y,z ) \in (\Kx)^3, (xyz^{-1})^{pc} =1 \right\}.
  \]
  As a set,
  this is the disjoint union of
  \begin{align}\label{eq_fib_comp}
    \left\{ [s : t : \zeta^k st : 1 : 1 : 1 : 1] \in X_A \subset \PP^6 \, | \, (s,t ) \in (\Kx)^2 \right\}
  \end{align}
  for $0 \leq k \leq c-1$, where $\zeta \in \Kx$ is a primitive $c$-th root of unity.
  Since the closure of each \ref{eq_fib_comp} is projectively equivalent to $X_{A'}= \PP^1 \times \PP^1 \subset \PP^3$,
  (i) and (iii) are satisfied.
\end{ex}

Next, we consider how to construct $X$ for a given integer $r > 0$ such that the rank of the Gauss map of $X$
is equal to $r$.
From the following remark, we need to assume $r \neq 1$ if $\chara\kk = 2$.

\begin{rem}\label{rem_rk_in_char2}
  In characteristic $2$,  it is known that
  the rank of the Gauss map of any projective variety $X \subset \PN$ \emph{cannot} be equal to $1$.
  In addition, if $X$ is a hypersurface, then the rank of the Gauss map is \emph{even}.
  The reason is as follows.
  
  Let $X \subset \PN$ be a projective variety in $\chara\kk = 2$,
  and let $x \in X$ be a general point.
  As in \cite[\textsection{}2]{fukaji2010},
  choosing homogeneous coordinates on $\PN$, we may assume that
  $X$ is locally parameterized at $x = [1:0:\dots:0]$
  by $[1:z_1: \dots: z_n: f_{n+1}: \dots: f_{N}]$,
  where $z_1, \dots, z_n$ form a regular system of parameters of $\sO_{X,x}$, and
  $f_{n+1}, \dots, f_N \in \sO_{X, x}$.
  Then
  $\rk d_{x}\gamma$
  is equal to the rank of the $n \times (n(N-n))$ matrix
  \begin{equation}\label{eq:Hessians}
    \begin{bmatrix}
      H(f_{n+1}) & H(f_{n+2}) & \cdots & H(f_{N})
    \end{bmatrix},
  \end{equation}
  where
  $H(f) :=
  [\partial^2 f / \partial z_i \partial z_j]_{1 \leq i,j \leq n}
  $
  is the Hessian matrix of a function $f$.
  Assume that $\rk (\gamma)$ ($= \rk d_x \gamma$) is nonzero.
  Then the matrix \ref{eq:Hessians} is nonzero; in particular,
  one of the Hessian matrix $H(f_{k})$ is nonzero.
  Since $\chara\kk = 2$, we have $\partial^2 f_k / \partial z_i\partial z_i = 0$,
  i.e., the diagonal entries of $H(f_{k})$ are zero. Hence 
  some $\partial^2 f_k / \partial z_i\partial z_j$ with $i \neq j$ must be nonzero.
  Therefore $H(f_{k})$ has $2\times 2$ submatrix
  \[
  \begin{bmatrix}
    \partial^2 f_k / \partial z_i\partial z_i & \partial^2 f_k / \partial z_j\partial z_i
    \\
    \partial^2 f_k / \partial z_i \partial z_j & \partial^2 f_k / \partial z_j\partial z_j
  \end{bmatrix}
  =
  \begin{bmatrix}
    0 & \partial^2 f_k / \partial z_i\partial z_j
    \\
    \partial^2 f_k / \partial z_i \partial z_j & 0
  \end{bmatrix},
  \]
  whose determinant is nonzero. This implies that $\rk (\gamma) \geq 2$.

  Now assume that $X \subset \PN$ is a hypersurface.
  Then $X$ is locally parametrized by $[1:z_1:\dots:z_n:f_{n+1}]$,
  and hence $\rk d_x \gamma = \rk H(f_{n+1})$.
  Since $\chara\kk = 2$, the symmetric matrix $H(f_{n+1})$ is skew-symmetric,
  whose rank is even
  (for example, see \cite[\textsection{}5, n$^\text{o}$1, Corollaire 3]{bourbaki}).

\end{rem}

\begin{thm}\label{cor_rk&fiber}
  Assume $p = \chara \KK > 0$.
  Let $A'$ be a finite subset of a free abelian group $M'$ with $\Aff (A') = M'$.
  Let $ r,c > 0$ be positive integers such that  $(p,r) \neq(2,1)$ and $c$ is coprime to $p$.
  Assume that positive integers $n, N$ satisfy
  \[
  n \geq \max \{  (\# A' -1) +r ,  \rk (M') + r +1\}
  \]
  and
  \begin{align}\label{cond_N}
    N \geq  \left\{ 
      \begin{array}{cl} 
        2n-\rk(M')-r+1 & \text{if } p \geq 3, \text{ or } \ p=2, r : \text{even} ,\\ 
        2n-\rk(M')-r+2 & \text{if } p=2, r : \text{odd} .\\ 
      \end{array} \right.
  \end{align}
  Then there exists a finite subset $A \subset M:= \Z^{n}$
  with $\Aff (A)= M$ and $\# A =N+1$
  such that the Gauss map $\gamma$ of $X_A \subset \PN$ satisfies the following conditions:
  \begin{enumerate}[\normalfont(i)]
  \item (The closure of) each irreducible component of a general fiber of $\gamma$
    is projectively equivalent to $X_{A'}$.
  \item The rank of $\gamma$ is equal to $r$.
  \item The number of the irreducible components of a general fiber of $\gamma$ is equal to $c$.
  \end{enumerate}
\end{thm}

\begin{proof}
  We set $n' = \rk(M')$ and  $\sAdm = \# A'-1$,
  and $A'=\{ u'_0,\ldots, u'_{\sAdm}\}$.
  Let $e_1, \ldots, e_{n}$ be the standard basis of $M=\Z^{n}$.
  Without loss of generality,
  we may assume that $u'_0=0$.
  As in the proof of \autoref{cor_im&fiber},
  we define a surjective group homomorphism $\pi : \Z^n \arw M' $
  by $\pi(e_i) = u'_i$ for $1 \leq i \leq \sAdm$ and $\pi(e_i)=0$ for $\sAdm+1 \leq i \leq n$.
  Since $\ker \pi \simeq \Z^{n-n'}$ and $e_{\sAdm+1} , \ldots, e_{N' + r} \in \ker \pi$
  (note that $\sAdm +r = (\# A' -1) +r \leq n$ holds by assumption),
  there exist
  \[
  f_1, \ldots, f_{n-n'-r} \in \ker \pi
  \]
  such that $e_{N'+1} , \ldots, e_{N'+r}, f_1, \ldots, f_{n-n'-r} $
  form a basis of $\ker \pi$.
  By assumption,
  $n-n' -r = n - \rk(M') -r \geq 1$ holds.

  \vspace{1mm}
  First,
  we consider the case
  when $N$ is equal to the right hand side of \ref{cond_N}.
  Set
  \[
  A = C \cup D \subset M,
  \]
  where
  \begin{align*}
    C &= \{ e_1,\ldots,e_n, 0, c pf_1, p f_2, \ldots, p f_{n-n'-r}\},
    \\
    D &= \left\{ 
      \begin{array}{cl} 
        \{e_{\sAdm+1} + \cdots + e_{\sAdm+r} \} & \text{for } r \not \equiv 1 \text{ mod } p ,\\ 
        \{- e_{\sAdm+1} - \cdots - e_{\sAdm+r} \} & \text{for } r \equiv 1 , r \not \equiv - 1 \text{ mod } p ,\\ 
        \{e_{\sAdm+1} + e_{\sAdm+2}, e_{\sAdm+2} + \cdots + e_{\sAdm+r} \} & \text{for }  p=2, r : \text{odd}, r \geq 3.
      \end{array} \right.
  \end{align*}
  By a similar argument in the proof of \autoref{cor_im&fiber} and by Examples \ref{prep-cor_rk&fiber}, \ref{ex_char2},
  we have
  \begin{align}\label{eq_B-B}
    \langle B-B \rangle = \bigoplus_{i=\sAdm+1}^{\sAdm+r}  \Z e_{i} \oplus \Z cp f_1 \oplus \bigoplus_{j=2}^{n-n'-r}  \Z p f_j  .
  \end{align}
  Hence 
  $\langle B-B \rangle_{\R} \cap M= \ker \pi$.
  Since $A' = \pi (A)$,
  (i) and (iii) in this corollary follows as in \autoref{cor_im&fiber}.

  Since $e_{N'+1} , \ldots, e_{N'+r}, f_1, \ldots, f_{n-n'-r} $ form a basis of $\ker \pi$,
  we have $\langle B-B \rangle_{\KK} =\bigoplus_{i=\sAdm+1}^{\sAdm+r}  \KK e_{j} $.
  Thus the rank of $\gamma $ is $\dim \langle B-B \rangle_{\KK} = r$ by \autoref{cor_rk&deg}.

  \vspace{3mm}
  Next, we consider any integer $N$ satisfying the inequality~\ref{cond_N}.
  We take
  a finite subset $E$ of the right hand side of \ref{eq_B-B} such that
  $N = \# (C \cup D \cup E) -1$
  for the above $C$ and $D$.
  Set
  $A := C \cup D \cup E \subset M$.
  Since
  $E$ is contained in the right hand side of \ref{eq_B-B}, the subgroup
  $\langle B-B \rangle $ for this $A$ is the same as \ref{eq_B-B}.
  Hence the assertion follows. \end{proof}

\end{document}